\documentclass[12pt]{amsart}

\usepackage{graphicx}

\usepackage{amssymb,amsmath,amsthm}
\usepackage{esint}
\usepackage{mathrsfs}
\usepackage{comment}
\usepackage{bbm,dsfont,bm}
\usepackage{cite}
\usepackage{color}
\usepackage[left=1.7cm, right=1.5cm, top=2cm]{geometry}

\theoremstyle{definition}

\newtheorem{theorem}[equation]{Theorem}
\newtheorem{lemma}[equation]{Lemma}

\newtheorem{definition}[equation]{Definition}

\newtheorem{remark}[equation]{Remark}
\newtheorem{proposition}[equation]{Proposition}

\usepackage{graphicx}

\newcommand{\norm}[1]{{\left \Vert{#1}\right \Vert}}

\usepackage{amsmath,amsthm,amssymb,enumerate, fancyhdr}
\usepackage[english]{babel}
\usepackage[square,sort,comma,numbers]{natbib}
\usepackage[backref=page]{hyperref}

\usepackage{thmtools}

\newtheorem*{theorem*}{Theorem}

\newcommand{\Rbb}{\mathbb{R}}

\newcommand{\ebf}{\mathbf{e}_{3}}

\numberwithin{equation}{section}

\begin{document}

\title[Time-periodic weak solutions for FSI]{Time-periodic weak solutions for an incompressible Newtonian fluid interacting with  an elastic plate.}

\author{Claudiu M\^{i}ndril\u{a}}
\address{Department of Analysis, Faculty of Mathematics and Physics, Charles University,
Sokolovsk\'{a} 83, 18675, Prague, Czech Republic }
\email{mindrila@karlin.mff.cuni.cz}

\author{Sebastian Schwarzacher}
\address{Department of Analysis, Faculty of Mathematics and Physics, Charles University,
Sokolovsk\'{a} 83, 18675, Prague, Czech Republic }
\email{schwarz@karlin.mff.cuni.cz}


\date{\today}



\begin{abstract}
Under the action of a time-periodic external force we prove the existence of at least one time-periodic weak solution for the interaction between a three-dimensional incompressible fluid, governed by the Navier-Stokes equation and a two dimensional elastic plate. The challenge is that the Eulerian domain for the fluid changes in time and is a part of the solution. We introduce a two fixed-point methodology: First we construct a time-periodic solutions for a given variable time-periodic geometry. Then in a second step a (set-valued) fixed point is performed w.r.t.\ the  geometry of the domain.  
The existence relies on newly developed a-priori estimates applicable for both coupled and uncoupled variable geometries.  Due to the expected weak regularity of the solutions such Eulerian estimates are unavoidable. Note in particular, that only the fluid is assumed to be dissipative; the here produced a-priori estimates show that its possible to exploit the dissipative effects of the fluid also for the solid deformation. The existence of time-periodic solutions for a given geometry is valid for arbitrary large data. The existence of periodic coupled solutions to the fluid-structure interaction is valid for all data that exclude a self-intersection a-priori.
\end{abstract}


\maketitle


\section{Introduction}

\noindent
Periodic solution naturally appear in the context of fluid-structure interactions. Applications are from various disciplines. Examples are from engineering (turbines, steam engines, oscillating bridge-decks) or biology (blood-flow). The latter one is the model example for this article: We consider an elastic wall that interacts with an incompresible fluid modeled by the unsteady Navier-Stokes equations in 3D. We seek for time-periodic  weak  solutions.  While many results for elastic plates or shells interacting with fluids have been shown in the last decades, the relevant regime of time-periodic solutions has rarely been touched. For an overview on applications and on the state of the art for such fluid-structure interactions we refer to the recent survey~\cite{Canic20} and the references therein. 

The present article aims to fill this gap by providing a methodology that allows for the construction of time-periodic weak solutions.  Our main result (Theorem~\ref{thm:main}) is the existence of at least one  time-periodic  weak solution for given periodic data. The constraints on the data are given only to exclude self-intersections of the elastic structure a-priori. Observe that in three space dimensions and for large data the frame of weak solutions is unavoidable (due to the open problem of regularity for Navier-Stokes equations). However in this setting energy estimates do {\em not suffice} to construct periodic solutions. Hence the key step is a new  regularity estimate. It is a non-trivial estimates of the displacement variable of the elastic wall, which is coupled to the fluid velocity via its time derivative only.
\subsection{State of the art}
The study of the  interaction between fluids and solids from a mathematical point of view has received a lot of interest in the last years. Several difficulties occur when studying these interactions. This includes among others the changes of the Eulerian domain in time, the choice of appropriate coupling (boundary) conditions or the respective concepts of weak solutions. The interaction with elastic solids, in particular, turned out  to be particularly difficult  due to the mixed parabolic-hyperbolic nature of the system, see also \cite{Bou05, RV14} and the references therein. It seems that due to the variable-in-time Eulerian geometry the matter of periodic solutions in that frame-work seems rather untouched.

In contrast, the literature for cylindrical time-space domains the existence of periodic solutions for incompressible Navier-Stokes systems is vast and we refer to \cite{GK18} where several results are surveyed.  A first existence result was shown in~\cite{Pr63} . Here, the periodicity is obtained at the discrete (Galerkin) level, by solving the corresponding ODE with a time-periodic condition; the latter objective is achieved by noticing that such a solution appears as a fixed point of the function that maps initial values to the endpoint values of the corresponding solutions.
This strategy has proved to be quite robust and actually it was  recently  employed in \cite{Ab20} in order to prove weak time-periodic solutions for a class of  non-Newtonian fluids.
The so-far available literature on fluid-structure interactions seems to be restricted to rigid body motions moving periodically inside a Navier-Stokes fluid. See the seminal works by Galdi at all~\cite{GS06, G13, G20, G20-2} and the references therein. See also \cite{BGM19}.
\\
{Further, we wish to mention the following recent works, where a similar set up is considered, but strong solutions are studied \cite{CK21,FKTW21}.
}

As indicated before we are concerned with \emph{ the interaction of a  viscous and incompressible fluid interacting with an elastic solid situated and being part of  the boundary of the fluid domain. }
The first existence result for {\em weak solutions} to the Cauchy problem of this subclass appeared in \cite{Gr05, Gr08}; there, the existence was proved for the case of a 3D/2D coupling of a viscous and incompressible fluid interacting  with an \emph{elastic plate} and resp.~with an \emph{elastic shell}. The motion of the fluid and shell is assumed to happen in the normal direction and the initial geometry has a \emph{flat} moving part of the boundary, where the shell is attached.   The next steps towards non-flat moving boundary and a linearized Koiter-type shell was considered in \cite{LR14},  while the more recent work \cite{MS19} deals with non-linear models of Koiter shells in the spirit of \cite{Ci05}. While this paper follows the strategy of fixed point methods there are also more constructibe approaches based on the A.L.E. method, see~\cite{MC13}. Finally let us mention recent progress in compressible and/or heat conducting fluids and shells~\cite{BS18,TW20,BS21}.

There are also many results tackling the existence/uniqueness/regularity of  \emph{strong solutions}. Among them let us refer to a recent work discussing the existence of strong time-periodic solutions within a 2D/1D coupling with a visco-elastic beam~\citep{Cas19}. There an abstract semi-group theory is used that is suitable for strong solutions that seems not to be suitable for the framework of weak solutions. We also point out the recent work \cite{MRR20} where a  maximal-in-time well possedness result of strong solutions for an incompressible fluid interacting with a damped shell  situated on the boundary is proved.
Let us also mention the recent  and excellent survey \cite{Canic20} where  besides a vast  surveryed literature also several current open problems  of   \emph{fluid-structure interation} were formulated.

\subsection{Moving domains}\label{subsection:domain}
Throughout this paper we consider a non-empty, open and bounded domain $\Omega \subset \mathbb{R}^{3}$ of class $C^{4}$  whose  boundary $ \partial \Omega=M \cup \Gamma$ consists of two parts:  $
\Gamma$ -the fixed part, and   $M$-the compact complement of $\Gamma$ which represents the moving part. Since we consider the plate case here, we assume that $M$ is flat
\[
 M=\omega\times\left\{ 0\right\} 
\]
where $\omega \subset \mathbb{R}^{2}$ is an open and  bounded Lipschitz domain. Moreover we assume that the {\em outer normal} of $\Omega$ at $M$ is $\ebf$.
On the moving part we attach an elastic shell whose displacement is denoted by $\eta : \omega \mapsto [a,b]$ and $\eta=0$ on $\partial \omega$ where $a,b\in \Rbb$. During this work we shall assume that the elastic deformation is parametrised as a graph and that it will change in {\em $\ebf$ direction only}.  
We define $x'=(x_1,x_2)$ and $z=x_3$.
\[
\phi_{\eta}(x',0)=(x',\eta(x'))\text{ for } (x',0)\in M\text{ and }\phi_{\eta}(x',z)=(x',z)\text{ on }\Gamma,
\]
and 
\begin{equation}\label{eqn:moving-domain}
\Omega_{\eta}:=(\Omega\setminus \left\{ (x',z) : 0>z>\eta\left(x'\right)\right\} ) \cup \left\{ (x',z) : 0\leq z<\eta\left(x'\right)\right\}.
\end{equation}
this implies by choosing $a,b$ appropriately that $\phi_{\eta}:\partial\Omega\to \partial\Omega_{\eta}$ is a homeomorphism.

Assuming that $\eta\in C^k_0(\omega)$ it is possible to extend $\phi_{\eta}:\Omega\to\Omega_{\eta}$, such that it is a $C^k$-diffeomorphism for $k\in\{0,1,2,3,4\}$.  

To simplify the notation, we assume in the following:
\begin{align}
\label{eq:domain}
\text{If } \norm{\eta}_\infty\leq \kappa\text{ and }\eta\in C^k_0(\omega)\text{, then }\phi_\eta \text{ is a $C^k$ diffeomorphism}.
 \end{align}
By denoting 
\begin{equation}
I:=(0,T)\ \text{and} \  I \times \Omega_{\eta} := \bigcup_{t\in I} \left\{ t\right\} \times\Omega_{\eta\left(t\right)}
\end{equation}  we can now introduce the Lebesgue and Sobolev spaces adapted to the moving domains $\Omega_{\eta}$. Namely,  for $1 \le p, \ q \le \infty$ we have
\begin{align*}
L^{p}\left(I;L^{q}\left(\Omega_{\eta}\right)\right):= & \left\{ \mathbf{v}\in L^{1}\left(I\times\Omega_{\eta}\right):\mathbf{v}\left(t,\cdot\right)\in L^{q}\left(\Omega_{\eta\left(t\right)}\right)\ \text{for a .e.}\ t\in I,\ \left\Vert \mathbf{v}\left(t,\cdot\right)\right\Vert _{L^{q}\left(\Omega_{\eta\left(t\right)}\right)}\in L^{p}\left(I\right)\right\} \\
L^{p}\left(I;W^{1,q}\left(\Omega_{\eta}\right)\right):= & \left\{ \mathbf{v} \in L^{1}\left(I\times\Omega_{\eta}\right):\mathbf{v}\left(t,\cdot\right)\in W^{1,q}\left(\Omega_{\eta\left(t\right)}\right)\ \text{for a .e.}\ t\in I,\ \left\Vert \mathbf{v}\left(t,\cdot\right)\right\Vert _{W^{1,q}\left(\Omega_{\eta\left(t\right)}\right)}\in L^{p}\left(I\right)\right\} .
\end{align*}
This spaces admit  the  usual \emph{trace} and \emph{extension} operators as long as $a<\eta<b$ and as long as $\eta\in C^{0,1}$; see \cite[Lemma 1]{Gr05} for further details.

\subsection{The Koiter elastic energy}\label{subsection:Koiter}
With the previous notations we can now introduce 
 the \emph{change of metric tensor}
\begin{equation}\label{eqn:tensor-metric}
\mathbb{G}_{ij}\left(\eta\right):=\partial_{i}\eta\cdot\partial_{j}\eta,\quad i,j=1,2
\end{equation}
and, after introducing the normal vector $\nu_{\eta}$ to the deformed surface via
\begin{equation}
 \nu_{\eta}:=\partial_{1} \phi_{\eta}\times\partial_{2} \phi_{\eta}
\end{equation}
 the \emph{change of curvature tensor} as
\begin{equation}\label{eqn:tensor-curvature}
\mathbb{R}_{ij}^{\sharp}:=\partial_{ij}\eta, \quad i,j=1,2.
\end{equation}

All-in-all,  if  $h$ denotes the thickness of the shell,  then the \emph{Koiter energy} of the plate reads as follows: 
\begin{equation}\label{eqn:Koiter-energy}
K\left(\eta\right):=K\left(\eta,\eta\right):=\frac{h}{2}\int_{\omega}\mathbb{A}:\mathbb{G}\left(\eta\right)\otimes\mathbb{G}\left(\eta\right)dx'+\frac{h^{3}}{6}\int_{\omega}\mathbb{A}:\mathbb{R}^{\sharp}\left(\eta\right)\otimes\mathbb{R}^{\sharp}\left(\eta\right)dx'.
\end{equation}
Here
\begin{equation}\label{eqn:lame}
\mathbb{A}^{ijkl}=\frac{4\lambda_s u}{\lambda_s+2\mu_s}\mathbb{A}^{ij}\mathbb{A}^{kl}+4\mu_s\left(\mathbb{A}^{ik}\mathbb{A}^{jl}+\mathbb{A}^{il}\mathbb{A}^{jk}\right)
\end{equation}
is a fourth-order tensor known as \emph{the elasticity (or stiffness) tensor},  and $\lambda_s,  \mu_s$ are the Lam\'{e} coefficients; see \cite[p. 162]{Ci05}.
Please observe that in the flat geometry $K$ is coercive in $H^{2}$.  More precisely, there exists a constant $c_0=c(\lambda_s,\mu_s, h)$ such that  
\begin{equation}\label{eqn: coercivity-K}
K\left(\eta\right)\ge c_0\left\Vert \nabla^{2}\eta\right\Vert _{L^{2}\left(\omega\right)}^{2}\quad\forall\eta\in H_{0}^{2}(\omega)
\end{equation}

\subsection{The coupled system}
We assume that the domain $\Omega_{\eta(t)}$ is filled  at any time $t\in (0,T)$ an \emph{incompressible} fluid of viscosity  $\mu$ and density $\rho_f$.  Its velocity flow denoted by
\[\mathbf{u}: I \times \Omega_{\eta} \mapsto \mathbb{R}^{3}
\] solves the  system 
\begin{equation}\label{eqn:Navier}
\begin{cases}
\rho_f(\partial_{t}\mathbf{u}+\left(\mathbf{u}\cdot\nabla\right)\mathbf{u})=\text{div}\sigma\left(\mathbf{u},p\right)+f & \text{in}\ I\times\Omega_{\eta(t)}\\
\text{div}\left(\mathbf{u}\right)=0 & \text{in}\ I \times\Omega_{\eta(t)}
\end{cases}
\end{equation}
where $\sigma(\mathbf{u},p)$ is the \emph{Cauchy stress tensor}  
\begin{equation}
\sigma(\mathbf{u},p)=-p I_{3}+2\mu D(\mathbf{u})
\end{equation}
where $D (\mathbf{u})$ denotes the symmetric gradient
\begin{equation}
D(\mathbf{u}) :=\frac{\nabla \mathbf{u} +\nabla \mathbf{u} ^{T}}{2}.
\end{equation}
The density of the external forces is given by
\begin{equation}
\mathbf{f}:\mathbb{R}\times\mathbb{R}^{3}\mapsto\mathbb{R}^{3},\ \mathbf{f}\left(t,\cdot\right)=\mathbf{f}\left(t+T,\cdot\right)\ \text{for a.e.} \ t\in I
\end{equation}
  Now, concerning the elastic shell, we denote its density by $\rho_s$, its thickness by $h$ and its  displacement w.r.t. the reference configuration by
\begin{equation}
\eta: I \times  \omega \mapsto \mathbb{R}.
\end{equation}
Then $\eta$  solves the following  hyperbolic system 
\begin{equation}\label{eqn:Koiter}
\begin{cases}
h\rho_s \partial_{tt}\eta+K^{\prime}\left(\eta\right)=g+\mathbf{F}\cdot\mathbf{e}_{3} & \text{in}\ I \times\omega\\
\eta=\nabla\eta=0 & \text{on}\ I \times\partial\omega
\end{cases}
\end{equation}
where $K^{\prime}(\eta)$ is the $L^{2}$-gradient of the Koiter energy $K$.
Here $\mathbf{F}$ denotes  the force exerted by the fluid on the shell, that is 
\begin{equation}
\mathbf{F}:=-\sigma\left(\mathbf{u},p\right)\left(t,\phi_{\eta}\left(t,x'\right)\right)\nu_{\eta}\cdot\mathbf{e}_{3},\ \left(x'\right)\in\omega
\end{equation}
 which is the stress tensor evaluated in normal direction.
 The distribution of external forces acting on the shell is denoted by 
 \begin{equation}
 g:\mathbb{R}\times\omega\mapsto\mathbb{R},\ g\left(t,\cdot\right)=g\left(t+T,\cdot\right)\ \text{for a.e.}\ t\in I
 \end{equation}

Regarding the boundary conditions, we assume that the fluid and the shell move with the same speed which means we assume the \emph{no-slip boundary condition}  $\text{tr}_{\eta}\mathbf{u}=\partial_{t}\eta\mathbf{e}_{3}\ $, or, more explicitly
\begin{equation}\label{eqn:no-slip-bdry-cond}
\mathbf{u}\left(t,x',\eta\left(t,x'\right)\right)=\left(0,0,\partial_{t}\eta\left(t,x'\right)\right),\quad x'\in\omega\text{ and }\mathbf{u}=0\text{ in }\Gamma.
\end{equation}

  the full system that will be studied reads as follows:

\begin{equation}\label{eq:system}
\begin{cases}
\rho_{f}(\partial_{t}\mathbf{u}+\left(\mathbf{u}\cdot\nabla\right)\mathbf{u})=\text{div}\sigma\left(\mathbf{u},p\right)+\mathbf{f} & \text{\ in}\ I\times\Omega_{\eta}\\
\text{div}\mathbf{u}=0 & \text{\ in}\ I\times\Omega_{\eta}\\
\mathbf{u}=0 & \ \text{on\ }I\times\left(\partial\Omega\setminus M\right)
\\
\mathbf{u}\circ\phi_{\eta}=\partial_{t}\eta\mathbf{e}_{3} & \ \text{in}\ I\times\omega
\\
\rho_{s}h\partial_{tt}\eta+K^{\prime}\left(\eta\right)=-\sigma\left(\mathbf{u},p\right)\left(t,\phi_{\eta\left(t\right)}\left(t,y\right)\right)\nu_{\eta}\cdot\mathbf{e}_{3}+g & \ \text{in}\ I\times\omega\\
\eta=\nabla\eta=0 & \text{\ on\ }I\times\partial\omega\\
\eta(0,\cdot)=\eta(T,\cdot),\ \partial_{t}\eta(0,\cdot)=\partial_{t}\eta(T,\cdot) & \text{\ in}\ \omega
\\
\mathbf{u}(0,\cdot)=\mathbf{u}(T,\cdot) & \ \text{in}\ \Omega_{\eta\left(0\right)}.
\end{cases}
\end{equation}
In order to reduce the notation we assume (without loss of generality) that $\rho_{f}=\rho_{s}h=\mu=1$  from now on .Throughout this work, we shall denote the surface measure of $\omega$ by $dx'$.

\begin{remark}[On the restrictions of the reference geometry]
\label{rmk:int-eta-constant-time}
Note that since $\mathbf{e}_{3}= (0,0,1)$ , from \eqref{eqn:no-slip-bdry-cond} and from the incompresiblity condition $\text{div} \mathbf{u}=0$ we find  that  the  volume of the whole cavity  $\int_{\omega} \eta(t,x)dx'$  has to remain constant. That is
\begin{equation}
\label{eq:mean}
\int_{\omega}\eta\left(t,\cdot\right)dx'=:m\text{ for all }t\in I.
\end{equation}
In the time-periodic setting this mean value $m$ is an independent constraint that can be given a-priori.
In fact, the constant mean-value is a notorious issue in fluid-structure interaction: there is a coupling of a deforming and potentially stretching solid with an incompressible fluid that conserves volume. If the prescribed direction of deformation is non orthogonal (or more explicit if the reference geometry has a non-zero Gauss curvature) equation \eqref{eq:mean} becomes non-linear w.r.t. the deformation (along the reference coordinate direction). That is the reason why in the theory for the existence of incompressible fluids the non-orthogonal deformations where an open problem for some time. The existence for the Cauchy problem was only shown in~\cite{LR14} by introducing respective non-linear correctors. We wish to emphasize that it is this key difficulty that reduces the set-up of this paper to flat geometries and excludes non-flat Koiter shells.
\end{remark}

\subsection{The main result}

The main theorem of this paper  is the following\footnote{For the precise defintion of the function spaces please see Subsection~\ref{ssec:defweaksol}}
\begin{theorem}\label{thm:main} 

There exists a constant $C_{0}=C_{0}(c_{0},\Omega,\alpha, M,\kappa,T)$    such that  for any $m \in \mathbb{R}$ and any given forces
\[
\left(\mathbf{f},g\right)\in\left(L_{\text{per}}^{2}\left(0,T\right);L^{2}\left(\mathbb{R}^{3}\right)\right)\times\left(L_{\text{per}}^{2}\left(0,T\right);L^{2}\left(\omega\right)\right)
\]
such that 
\[ m^2+
\int_{0}^{T}\int_{\mathbb{R}^{3}}\left|\mathbf{f}\right|^{2}dxdt+\int_{0}^{T}\int_{\omega}\left|g\right|^{2}dx'dt\le C_0
\]
 the problem \eqref{eq:system}  admits at least one weak time-periodic solution $(\mathbf{u},\eta)$ in the sense of Definition~\ref{def:weak-periodic-solution}, with
\begin{equation}\label{eqn:eta=m}
\int_{\omega}\eta(t,\cdot)dx'=m \text{ for all }t\in I.
\end{equation} 
Furthermore it holds that \begin{equation}\label{eqn:energy-estimate-1}
\left\Vert \mathbf{u}\right\Vert _{L^{\infty}\left(I;L^{2}\left(\Omega_{\eta}\right)\right)\cap L^{2}\left(I;W_{0,\text{div}}^{1,2}\left(\Omega_{\eta}\right)\right)}+\left\Vert \eta\right\Vert _{L^{\infty}\left(I;H_{0}^{2}\left(\omega\right)\right)\cap W^{1,\infty}\left(I;L^{2}\left(\omega\right)\right)}\le c\left(C_{0}^{2}+C_{0}\right)
\end{equation} with  $c$ depending on $\Omega$, $\kappa$, $c_0$, $T$  only. 
Additionally,    the following  energy inequality holds
\begin{equation}\label{eqn:energy-estimate-2}
\int_{0}^{T}\int_{\Omega_\eta\left(t\right)}\left|\nabla \mathbf{u}\right|^{2}dxdt\le\int_{0}^{T}\int_{\Omega_\eta\left(t\right)}\mathbf{f}\cdot \mathbf{u}dxdt+\int_{\omega}g\cdot\partial_{t}\eta dx'dt.
\end{equation}
\end{theorem}
%
%
\begin{remark}[The 2D/1D case]
Note that our proof can be adapted to the 2D coupling without of the fluid with a 1D coupling of the shell without any substantial change.
\end{remark}
\begin{remark}[On the constraints in Theorem~\ref{thm:main}]
\label{rmk:further-explanations-thickness-shell}
One critical and non-trivial fact about fluid-structure interactions (in general) is that the solution prescribes the geometry of the fluid domain.  Generally,  weak solutions can be constructed up to the point of self-intersection of the solid. The key obstacle to produce periodic solutions (where the initial condition is the unknown!) is to exclude self-intersection over the {\em full time-interval} a-priori. It is only this necessity that determines the size of the constraint $C_0$. Certainly the theorem hence provides the existence of periodic solutions under the assumption that the applied forces are small. However, since it is the deformation that restricts the geometric flexibility {\em and no other part of the solution construction} it is also possible to assume that the elastic resistance of the shell is large enough to withstand the given forces. Actually, as can be checked, the assumption on the constant $C_0$ are strongly connected to the $H^2$-coercivity of the shell $c_0$. This coercivity again is connected to the thickness of the elastic plate and respectively its resistance against deformations. Hence one consequence of the theorem above shows that periodic weak-solutions for arbitrary large periodic forces (and arbitrary large time-intervals) can be constructed as long as the structure is "thick" enough.  For some applications these assumptions seem to be appropriate. Indeed, the stronger the periodic forces (stemming from a pump for instance) are acting on the structure, the more resistant the structure (the pipe) should be. 
\end{remark}

\begin{remark}[On the assumed boundary values]
With regard of the importance of applications like pipelines or blood-vessels inflow/outflow boundary conditions are very natural to be considered. In that regard observe that the here provided existence theory can be addapted to periodic-in-space boundary conditions on some part of the domain without substantial effort (as was considered by many authors~\cite{Gr05,Canic20}). More general inflow and outflow conditions will be the subject to a follow-up-work of this paper.
\end{remark}

\begin{remark}[Key novelty]
Please observe that the a-priori estimate provided  in \eqref{eqn:energy-estimate-1} is not a trivial consequence of the energy estimate \eqref{eqn:energy-estimate-2}. Actually, we consider the estimate on the deformation to be the main technical contribution of the present paper. Formally this non-trivial estimate is explained in Subsection~\ref{ssec:a-priori estimates}.
\end{remark}

\subsection{Outline of the paper}
Our paper is structured as follows: in Section~\ref{sec:a-priori estimates} we introduce the periodic spaces and the concept of weak solutions for which we will show the existence. Before we motivate both by justifying formally how to obtain the necessary a-priori estimates which are \emph{independent} of the initial data.
In Section~\ref{sec:proof-of-main-result} we prove our main result.
We perform two fixed-points: One at the discrete (Galerkin) level, employing Sch\"{a}ffer's fixed point Theorem, which ensures the periodicity of $ \mathbf{u}_n$,  $\eta_{n}$,$ \partial_{t} \eta_{n}$ for a given geometry. At this level also the critical a-priori estimates will be derived that eventually will withstand all further analytical manipulations. The second fixed-point is performed at the continuous level for a regularized set-up w.r.t.\ the geometry. Here we employ the set-valued fixed point Theorem~\ref{thm: Kakutain} due to Kakutani-Glicksberg-Fan (since the uniqueness of solutions can not be guaranteed). Finally, the last step consists in letting $\varepsilon \to 0$ which provides our main result, Theorem~\ref{thm:main}. 
The paper ends with an Appenix--Section~\ref{section:appendix}--where we collect some known results, essentially needed in the proof.
\section{Preliminary}
 \label{sec:a-priori estimates}
 \subsection{Formal a-priori estimates} \label{ssec:a-priori estimates} 
 
 Assume during this subsection that all the involved functions  are smooth.
 We will obtain suitable a-priori estimates needed to prove our main result--Theorem~\ref{thm:main}.
 Let us first use the test function $(\mathbf{u},\eta)$. That is we  take the scalar product with $\mathbf{u}$  and $\partial_{t}\eta$ in \eqref{eqn:Navier} and in \eqref{eqn:Koiter} respectively. We add the resulting equations and after  recalling  the coupling $\mathbf{u} = \partial_{t} \eta \mathbf{e}_{3}$ on $\partial \Omega _{\eta}$ and the definition of $\mathbf{u}$ we get
\begin{equation}\label{eqn:formal-energy}
\frac{d}{dt}E\left(t\right)+\int_{\Omega_{\eta\left(t\right)}}\left|\nabla\mathbf{u}\right|^{2}dx=\int_{\Omega_{\eta\left(t\right)}}\mathbf{f}\cdot\mathbf{u}dx+\int_{\omega}g\partial_{t}\eta dx'\text{ for all }t\in\left[0,T\right]
\end{equation}
where we have denoted
 \[
E\left(t\right):=\frac{1}{2}\int_{\Omega_{\eta\left(t\right)}}\left|\mathbf{u}\left(t,\cdot\right)\right|^{2}dx+\frac{1}{2}\int_{\omega}\left|\partial_{t}\eta\left(t,\cdot\right)\right|^{2}dx'+K\left(\eta\left(t\cdot\right)\right),\ t\in\left[0,T\right].
\]
Now since $\mathbf{u}$ vanishes on $\Gamma \subset \partial \Omega _{\eta(t)}$ and  $\mathbf{u}=\partial_{t} \eta \cdot \mathbf{e}_{3}$ on $\partial \Omega_{\eta(t)}$, by using Poincare's inequality (see Theorem~\ref{thm:poincare})  and recalling the properties of the trace operator from Lemma~\ref{lm:trace} we obtain
\begin{equation}\label{eqn:formal1}
\int_{0}^{T}\int_{\Omega_{\eta\left(t\right)}}\left|\mathbf{u}\right|^{2}dxdt+\int_{\omega}\left|\partial_{t}\eta\right|^{2}dx'dt\le c (\Omega, \kappa)\int_{0}^{T}\int_{\Omega_{\eta\left(t\right)}}\left|\nabla \mathbf{u}\right|^{2}dxdt.
\end{equation}

On the other hand, using the periodicity of $\mathbf{u}, \eta, \partial_{t} \eta$ we also  have $E(0)=E(T)$ and thus
\begin{equation}\label{eqn:formal2}
\int_{0}^{T}\int_{\Omega_{\eta\left(t\right)}}\left|\nabla \mathbf{u}\right|^{2}dxdt=\int_{0}^{T}\int_{\Omega_{\eta\left(t\right)}}\mathbf{f}\cdot \mathbf{u}dx+\int_{\omega}g\partial_{t}\eta dx'dt.
\end{equation}
From \eqref{eqn:formal1} and \eqref{eqn:formal2} we obtain
\begin{equation}\label{eqn:formal3}
\int_{0}^{T}\int_{\Omega_{\eta\left(t\right)}}\left|\mathbf{u}\right|^{2}+\left|\nabla\mathbf{u}\right|^{2}dxdt+\int_{\omega}\left|\partial_{t}\eta\right|^{2}dx'dt\le cC\left(\mathbf{f},g\right).
\end{equation}

where $C(\mathbf{f},g)$ will always denote 
\begin{equation}\label{eqn:C(f,g)}
C\left(\mathbf{f},g\right):=\int_{0}^{T}\int_{\mathbb{R}^{3}}\left|\mathbf{f}\right|^{2}dxdt+\int_{0}^{T}\int_{\omega}\left|g\right|^{2}dx'dt\leq C_{0}-m^{2}.
\end{equation}

It is important to remark that we will obtain the periodic solutions as a fixed point of the mapping
 \[\left(\mathbf{u}_{0},\eta_{0},\eta_{1}\right)\mapsto\left(\mathbf{u}\left(T,\cdot\right),\eta\left(T,\cdot\right),\partial_{t}\eta\left(T,\cdot\right)\right)
 \] where $\left(u_{0},\eta_{0},\eta_{1}\right)$ is a set of initial values.\footnote{This approach is by now standard--see for example \cite{Pr63,GK18, Ab20} and the references therein.}
  This is why  we need estimates of $E(T)$ which are \emph{independent } on the  initial value $E(0)$.  
  
From the mean value theorem there is $t_0 \in [0,T]$ for which
\[E\left(t_{0}\right)=\frac{1}{T}\int_{0}^{T}E\left(t\right)dt=:\fint_{0}^{T}E\left(t\right)dt\]
and so by integrating the equation~\eqref{eqn:formal-energy} on $[t_0, t]$ and taking the supremum over $t$ we get that 
\begin{equation}\label{eqn:formal-K-prime}
\sup_{0\le t\le T}E\left(t\right)\le\fint_{0}^{T}E\left(t\right)dt+cC\left(\mathbf{f},g\right)\le cC\left(\mathbf{f},g\right)+\fint_{0}^{T}K\left(\eta\right)dt
\end{equation}
In order to estimate the last term we would like to test the shell equation with 
 $\eta$. In order to do that we seek a corresponding extension $\mathcal{E}(\eta)$ for the fluid equation such that $\text{div} \mathcal{E}(\eta)=0$ and $\text{tr} \mathcal{E}(\eta)\circ\phi_\eta=\eta \mathbf{e}_{3} $.  
Finding such a function $\mathcal{E}(\eta)$ is not immediate. But in \cite[Proposition 3.3]{MS19}, such a construction was performed, along with corresponding  estimates as long as the compatibility conditions with the solenoidality is satisfied.  We include the properties of the operator $\mathcal{E}(\eta)$ in Proposition~\ref{prop:solenoidal-extension-operator}, that can be found in the appendix. Roughly speaking   $\mathcal{E}(\eta) \sim \eta$ 
in  the sense that obeys the same estimates up to a constant depending on $\kappa, c_0, \Omega$. As mentioned the compatibility condition with the solenoidality of the extension has to be satisfied. 
For that we fix a smooth bump function
\begin{equation}\label{eqn:psi-bump-function-test-eta}
 \psi \in C_{0}^{\infty}(\omega;[0,1]) \  \text{with} \int_{\omega} \psi dx'=1 
\end{equation}
 
 and by recalling the definition of $m$ from Remark~\ref{rmk:int-eta-constant-time} we find  that the following couple 
\[
\left(\mathcal{F}_{\eta}\left(\mathcal{M}_{\eta}\left(\eta\right)\right),\mathcal{M}_{\eta}\left(\eta\right)\right)=\left(\mathcal{F_{\eta}}\left(\eta-m\psi\right),\eta-m\psi\right)
\]
is a test function for the weak formulation given in Definition~\ref{def:weak-periodic-solution}.
As we show below, this test-function does provide the following estimate 
\begin{equation}\label{eqn:formal-claim}
2\int_{0}^{T}K\left(\eta\right)dt\le\frac{1}{2}\sup_{0\le t\le T}E\left(t\right)+c\left(C_0+C_0^2\right),
\end{equation}
which implies the desired estimate by absorbing $\frac{1}{2}\sup_{0\le t\le T}E\left(t\right)$ to the left-hand side of \eqref{eqn:formal-K-prime}.

As mentioned above, the estimate \eqref{eqn:formal-claim} follows by using the extension of $\eta$ as a test-function in Definition~\ref{def:weak-periodic-solution}. We then find
\begin{align*}
2\int_{0}^{T}K\left(\eta\right)dt&=\int_{0}^{T}\left\langle K^{\prime}\left(\eta\right),\eta\right\rangle dt\le \int_{0}^{T}\left\langle K^{\prime}\left(\eta\right),\eta-\mathcal{M}_{\eta}\left(\eta\right)\right\rangle dt
+\int_{0}^{T}\int_{\omega}\partial_{t}\eta\partial_{t}\left(\mathcal{M}_{\eta}\left(\eta\right)\right)dx'dt
\\
&\quad +
  \int_{0}^{T}\int_{\Omega_{\eta\left(t\right)}}\mathbf{u}\cdot\partial_{t}\mathcal{F}_{\eta}\left(\mathcal{M}_{\eta}\left(\eta\right)\right)+\left(\nabla \mathbf{u}+\mathbf{u}\otimes \mathbf{u}\right):\nabla\mathcal{F}_{\eta}\left(\mathcal{M}\left(\eta\right)\right)dxdt
  \\
&\quad + \int_{0}^{T}\int_{\Omega_{\eta\left(t\right)}}\mathbf{f}\cdot\mathcal{F}_{\eta}\left(\mathcal{M}\left(\eta\right)\right)dxdt+\int_{0}^{T}\int_{\omega}g\left(\mathcal{M}_{\eta}\left(\eta\right)\right)dx'dt
 =:\sum_{k=1}^{7}I_{k}.
\end{align*}
Recalling now the properties of the divergence-free extension operator $\mathcal{F}_{\eta}$ from Proposition~\ref{prop:solenoidal-extension-operator}
we employ  H\"{o}lder's inequalilty,  the Sobolev embedding, the $H^2$-coercivity of $K$~\eqref{eqn: coercivity-K} and the estimate~\eqref{eqn:C(f,g)} to get 
\begin{align*}
I_{1} & \le m\left\Vert K^{\prime}\left(\eta\right)\right\Vert _{L_{t}^{2}L_{x}^{2}}\\
I_{2} & \le\left\Vert \partial_{t}\eta\right\Vert _{L_{t}^{2}L_{x}^{2}}^{2}\le cC_0\\
I_{3}+I_{4}+I_{5} & \le\left(\left\Vert \mathbf{u}\right\Vert _{L_{t}^{2}W_{x}^{1,2}}+\left\Vert \mathbf{u}\right\Vert _{L_{t}^{2}W_{x}^{1,2}}^{2}\right)\left\Vert \mathcal{F}_{\eta}\left(\mathcal{M}_{\eta}\right)\right\Vert _{W_{t}^{1,2}L_{x}^{2}\cap L_{t}^{\infty}W_{x}^{1,2}}
\\
 & \le c(\sqrt{C_0}+C_0)\left\Vert \mathcal{M}_{\eta}\right\Vert _{W_{t}^{1,\infty}L_{x}^{2}\cap L_{t}^{\infty}W_{x}^{2,2}}
 \\
 & \le c(\sqrt{C_0}+C_0)\sup_{t}\sqrt{E\left(t\right)}+c(C_0^2
+C_0)
 \\
I_{6}+I_{7} & \le\left(\left\Vert \mathbf{f}\right\Vert _{L_{t}^{2}L_{x}^{2}}+\left\Vert g\right\Vert _{L_{t}^{2}L_{x}^{2}}\right)\left(\left\Vert \mathcal{F}_{\eta}\left(\mathcal{M}_{\eta}\right)\right\Vert _{L_{t}^{2}L_{x}^{2}}+\left\Vert \mathcal{M}_{\eta}\right\Vert _{L_{t}^{2}L_{x}^{2}}\right)\\
 & \le c\sqrt{C_0}\sup_{t}\sqrt{E\left(t\right)} +c C_0
\end{align*}
All in all,  using Young's inequality we conclude that 
\[
\sum_{k=1}^{7}I_{k}\le \frac{1}{2}\sup_{t}E\left(t\right)+c(C_0+C_0^2).
\]
Hence we get \eqref{eqn:formal-claim} which together with \eqref{eqn:formal-K-prime} implies
 \begin{equation}\label{eqn:formal-final}
\sup_{0\le t\le T}E\left(t\right)\le c(C_0+C_0^2).
\end{equation}

\subsection{Spaces of time-periodic functions}
In the spirit of the  formal estimates from above it is natural to consier the following function spaces. 
 For  $T>0$ and a general (real) Banach space $X$ we define \[C_{\text{per}}^{\infty}\left(0,T;X\right):=\left\{ \varphi:\mathbb{R}\to X:\ \varphi\text{\ is}\ C^{\infty}\ \text{and }\varphi\left(t+T, \cdot\right)=\varphi\left(t,\cdot\right)\right\} \] and then one obtains for $k \in \mathbb{Z}$ and $1\le p\le \infty$  the usual periodic spaces \[L^{p}_{\text{per}}\left(0,T;X\right):=\overline{\left\{ f\in C_{\text{per}}^{\infty}\left(0,T;X\right)\right\} }^{\left\Vert \cdot\right\Vert _{p}}\quad W^{k,p}_{\text{per}}\left(0,T;X\right):=\overline{\left\{ f\in C_{\text{per}}^{\infty}\left(0,T;X\right)\right\} }^{\left\Vert \cdot\right\Vert _{k,p}}.\]
 Further we consider $C_{w,\text{per}}(I,X)$ as the periodic weakly continuous space.\footnote{We say that $x\in C_{w,\text{per}}(I,X)$ if for all $x^*\in C_{\text{per}}(\overline{I},X^*)$, $t\mapsto\langle x(t), x^*(t)\rangle$ is continuous in $\overline{I}$; this concept can naturally be extended to time-changing domains.}
 We denote  $I:=(0,T)$ the time interval, where $T$ is the time-period.
 
  In the following we will use spaces that change in time. Such spaces have been investigated carefully in the previous literature (see for instance~\cite{PhD11}) and hence the concept of time-periodic spaces can be extended in a straight forward manner. 
We have:
\begin{equation}\label{eqn:def-function-spaces}
\begin{aligned}V_{\eta}\left(t\right):= & \left\{ \mathbf{u}\in H^{1}\left(\Omega_{\eta\left(t\right)}\right)
:\text{div}\mathbf{u}=0 \text{ in } \Omega_{\eta(t)},\,\mathbf{u}=0 \text{ on } \Gamma \right\}
\\
V_{F,\text{per}}:= & \left\{ \mathbf{u}\in C_{w,\text{per}}\left(I;L^{2}\left(\Omega_{\eta}\right)\right)\cap L_{\text{per}}^{2}\left(I;V_{\eta}\left(t\right)\right)\right\} 
\\
V_{K,\text{per}}:= & C_{w,\text{per}}\left(\mathbb{R};H_{0}^{2}\left(\omega\right)\right)\cap W_{\text{per}}^{1,\infty}\left(\mathbb{R};L^{2}\left(\omega\right)\right)
\\
V_{S,\text{per}}:= & \left\{ \left(\mathbf{u},\eta\right)\in V_{F,\text{per}}\times V_{K,\text{per}};\mathbf{u}\left(t,\phi_{\eta}\left(t,\cdot\right)\right)=\partial_{t}\eta\left(t,\cdot\right)\mathbf{e}_{3}\left(\eta\left(t,\cdot\right)\right) \text{ in } \omega\right\} 
\\
V_{T,\text{per}}:= & \left\{ \left(\mathbf{q},\xi\right)\in 
C^1_{\text{per}}(I;V_\eta(t))\times V_{K,\text{per}}:\left(I\times\Omega_{\eta}\right),\mathbf{q}\left(t,\phi_{\eta}\left(t,\cdot\right)\right)=\xi\left(t,\cdot\right)\mathbf{e}_{3}\left(\eta\left(t,\cdot\right)\right)\right\}.
\end{aligned}
\end{equation}
Note that $V_{S,\text{per}}$ is the space of time-periodic solutions and $V_{T,\text{per}}$ is the space of time-periodic test-functions.  
In particular $(\mathbf{u},\eta)\in V_{S,\text{per}}$ implies that $\mathbf{u}\equiv 0$ over $(0,T)\times \Gamma$ as well as $\left(\mathbf{q},\xi\right)\in V_{T,\text{per}}$ implies that $\mathbf{q}\equiv 0$ over $(0,T)\times \Gamma$.

{Further, if the geometry is {\em decoupled} from the solid deformation we shall use the analogous notations $V_{T,\text{per}}^{\delta},V_{S,\text{per}}^{\delta}$  whenever we deal with functions $(u, \eta)$ for which 
\[
\mathbf{u}\left(t,\phi_{\delta(t)}\right)=\partial_{t}\eta\mathbf{e}_{3}
\] and $\delta\in C\left(I\times\omega\right),\ \left\Vert \delta\right\Vert _{L^{\infty}\left(I\times\omega\right)}\leq \kappa$, a here given function defining the time-changing domain.}
\begin{remark} \label{remark: Holder-cont}
We have $V_{K,\text{per}}\hookrightarrow C^{0,1-\theta}\left(\overline{I};C^{0,2\theta-1}\left(\omega\right)\right)$, and therefore  the displacement of the shell is H\"{o}lder continuous in time.
 Indeed, using the embeddings $H^{2}\left(\omega\right)\hookrightarrow H^{2\theta}\left(\omega\right)\hookrightarrow C^{0,2\theta-1}\left(\omega\right)$ for $\theta \in (1/2,1)$  (since $\omega\in \mathbb{R}^2$). Consequently we find
  \begin{align*}
\left\Vert \eta\left(t\right)-\eta\left(s\right)\right\Vert _{C^{0,2\theta-1}\left(\omega\right)}	&\le c\norm{\eta\left(t\right)-\eta\left(s\right)}_{H^{2\theta}(\omega)}
\leq c\left\Vert \eta\left(t\right)-\eta\left(s\right)\right\Vert _{H^{2}\left(\omega\right)}^{\theta}\left\Vert \eta\left(t\right)-\eta\left(s\right)\right\Vert _{L^{2}\left(\omega\right)}^{1-\theta}
\\
	&\le c\left\Vert \eta\right\Vert _{L^{\infty}\left(I;H^{2}\left(\omega\right)\right)}^{\theta}\left\Vert \eta\right\Vert _{W^{1,\infty}\left(I;L^{2}\left(\omega\right)\right)}^{1-\theta}\left|t-s\right|^{1-\theta} .
\end{align*} In particular, the condition $\eta(0,\cdot)=\eta(T,\cdot)$   holds in a strong sense, in contrast with   the  conditions  $\mathbf{u}(0,\cdot)=\mathbf{u}(T,\cdot)$ and $\partial_{t} \eta(0,\cdot)=\partial_{t} \eta (T,\cdot)$ which  can only occur in a weak  sense (continuous in time  and taking values in   suitable  functions spaces endowed with a weak-topology), see Definition~\ref{def:weak-periodic-solution}.
\end{remark}
\subsection{The definition of weak time-periodic solutions}
\label{ssec:defweaksol}
Based on the a-priori estimates we can introduce the following  notion of \emph{periodic solution}. Let $(q, \xi ) \in V_{T,\text{per}} ^{\eta}$ be a test function which is smooth in time and space. 
First we notice that \begin{align*}
\int_{0}^{T}\int_{\Omega_{\eta\left(t\right)}}\partial_{t}\mathbf{u}\cdot\mathbf{q}dxdt= & \int_{0}^{T}\left(\frac{d}{dt}\int_{\Omega_{\eta\left(t\right)}}\mathbf{u}\cdot\mathbf{q}dx-\int_{\Omega_{\eta\left(t\right)}}\mathbf{u}\cdot\partial_{t}\mathbf{q}dx-\int_{\omega}\left(\partial_{t}\eta\right)^{2}\xi dx'\right)dt\\
= & -\int_{0}^{T}\int_{\Omega_{\eta\left(t\right)}}\mathbf{u}\cdot\partial_{t}\mathbf{q}dxdt-\int_{0}^{T}\int_{\omega}\left(\partial_{t}\eta\right)^{2}\xi dx'dt
\end{align*} where we have  used Reynolds' transport Theorem~\ref{thm:Reynolds}  and the fact that $\text{tr}_{\eta} \mathbf{u}\circ \phi_\eta=\partial_t\eta \mathbf{e}_{3}$ and also $\text{tr}_{\eta} \mathbf{q}\circ \phi_\eta= \xi \mathbf{e}_3$. On the other hand if we look at the convective term we see that
\begin{align*}\int_{0}^{T}\int_{\Omega_{\eta\left(t\right)}}\left(\mathbf{u}\cdot\nabla\right)\mathbf{u}\cdot\mathbf{q}dxdt= & \sum_{i=1}^{3}\int_{0}^{T}\int_{\Omega_{\eta\left(t\right)}}\partial_{i}\left(\mathbf{u}^{i}\mathbf{u}^{j}\right)\mathbf{q}^{j}-\partial_{i}\mathbf{u}^{i}\mathbf{u}^{j}\mathbf{q}^{j}dxdt\\
= & \sum_{i=1}^{3}\int_{0}^{T}\int_{\partial\Omega_{\eta\left(t\right)}}\mathbf{u}^{i}\mathbf{u}^{j}\mathbf{q}^{j}\nu^{i}dx'dt-\int_{0}^{T}\int_{\Omega_{\eta\left(t\right)}}\left(\mathbf{u}^{i}\mathbf{u}^{j}\right)\partial_{i}\mathbf{q}^{j}dxdt\\
= & \int_{0}^{T}\int_{\omega}\left(\partial_{t}\eta\right)^{2}\xi dx'dt-\int_{0}^{T}\int_{\Omega_{\eta\left(t\right)}}\left(\mathbf{u}\cdot\nabla\right)\mathbf{q}\cdot\mathbf{u}dxdt
\end{align*}
where we have used the coupling of $\mathbf{u}$ and $\partial_{t}\eta$ , the  $\text{div} \mathbf{u}=0$  condition and Stokes' Theorem. We also make use of the Korn's identity  from Lemma~\ref{lm:Korn} and get
 \[
\int_{\Omega_{\eta\left(t\right)}}\nabla \mathbf{u}:\nabla \mathbf{q}dx=2\int_{\Omega_{\eta\left(t\right)}}D\left(\mathbf{u}\right):D\left(\mathbf{q}\right)dx.\]
All in all we obtain the following definition.
\begin{definition}\label{def:weak-periodic-solution} A couple $(\mathbf{u},\eta) \in V_{S, \text{per}}$ with $\left\Vert \eta\right\Vert _{L^{\infty}\left(I\times\omega\right)}<\kappa$ is a \emph{ weak time-periodic solution} to \eqref{eq:system} if 

\begin{equation}\label{eq: weak-periodic-solution}
\begin{aligned}\int_{0}^{T}\int_{\Omega_{\eta\left(t\right)}}-\mathbf{u}\cdot\partial_{t}\mathbf{q}+\nabla\mathbf{u}:\nabla\mathbf{q}-\left(\mathbf{u}\cdot\nabla\right)\mathbf{q}\cdot\mathbf{u}dxdt+\int_{0}^{T}\left\langle K^{\prime}\left(\eta\right),\xi\right\rangle -\int_{\omega}\partial_{t}\eta\partial_{t}\xi dx'dt & =\\
\int_{0}^{T}\int_{\Omega_{\eta\left(t\right)}}\mathbf{f}\cdot\mathbf{q}dxdt+\int_{\omega}g\xi dx'dt\quad\forall\left(\mathbf{q},\xi\right)\in V_{T,\text{per}}^{\eta}.
\end{aligned}
\end{equation}
\end{definition}

\section{Proof of the main result}\label{sec:proof-of-main-result}

 \subsection{The decoupled system}
 \label{subsection:decoupled}
In the following we assume that  $\delta\in C\left(I\times\omega\right),\ \left\Vert \delta\right\Vert _{L^{\infty}\left(I\times\omega\right)}<\kappa$ is prescribing a periodic time-changing domain. This means in particular that $\delta(0,\cdot)=\delta(T, \cdot)$ and $\Omega_{\delta(t)}=\phi_{\delta(t)}(\Omega)$. 
Next we will regularize the space variables of the prescribed geometry $\delta$,  and therefore replace $\delta$ by $\mathcal{R}_{\varepsilon} \delta$, where $\varepsilon$ is the mollification parameter.  
We extend $\delta$ to $(-\infty, T] \times \omega$ by $\delta (t, \cdot)=\delta(0,\cdot)$ for $ t<0$. 

\begin{lemma}[\cite{LR14},\cite{BS18}]
\label{lm:regularizers}
There exists an operator \[\mathcal{R}_{\varepsilon}:C\left(\left[0,T\right]\times\omega\right)\mapsto C^{4}\left(\left[0,T\right]\times\omega\right)\] such that, 
\begin{enumerate}[(a)]
\item $\mathcal{R}_{\varepsilon}\delta\to\delta$ uniformly as $\varepsilon \to 0$;
\item $\mathcal{R}_{\varepsilon}:L^{2}\left(0,T;H_{0}^{2}\left(\omega\right)\right)\mapsto L^{2}\left(0,T;H_{0}^{2}\left(\omega\right)\right)$ and $\mathcal{R}_{\varepsilon}\delta\to\delta$ in $L^{2}\left(0,T;H_{0}^{2}\left(\omega\right)\right)$ as $\varepsilon \to 0$;
\item If $\partial_{t}\delta\in L^{p}\left(\left(0,T\right)\times\omega\right)$ then $\partial_{t}\mathcal{R}_{\varepsilon}\delta=\mathcal{R}_{\varepsilon}\partial_{t}\delta\to\partial_{t}\delta$ in $L^{p}\left(\left(0,T\right)\times\omega\right)$ as $\varepsilon \to 0$;
\item If $\delta\in C^{\gamma}\left(\left(0,T\right)\times\omega\right)$ for some $\gamma \in (0,1)$ then $\mathcal{R}_{\varepsilon}\delta\to\delta$ in $C^{\gamma}\left(\left(0,T\right)\times\omega\right)$ as $\varepsilon \to 0$;
\item $\left\Vert \mathcal{R}_{\varepsilon}\delta\right\Vert _{L^{\infty}\left(\left(0,T\right)\times\omega\right)}\le\left\Vert \delta\right\Vert _{L^{\infty}\left(\left(0,T\right)\times\omega\right)}$.
\end{enumerate}
\end{lemma}
We will also regularize functions $\mathbf{v} \in L^{2} ((0,T) \times \mathbb{R}^{3})$ by convolution with standard mollifying kernels. Thus, if $\psi_\varepsilon$ denote the standard time-space mollification kernels, we set \[\mathcal{R}_{\varepsilon}v\left(t,x\right):=\int_{\mathbb{R}^{3+1}}\psi_{k}\left(t-s,x-y\right)\chi_{\left(0,T\right)\times\Omega_{\mathcal{R}_{\varepsilon}\delta}}\left(s,y\right)\mathbf{v}\left(s,y\right)dsdy.\]
Observe, that the convective term can be linearized in such a way that the energy estimates are preserved via Reynolds' Transport Theorem. Indeed, observing that
  \[
\int_{\Omega_{\delta\left(t\right)}}\left(\mathbf{u}\cdot\nabla\right)\mathbf{u}\cdot\mathbf{q}dx+\int_{\Omega_{\delta\left(t\right)}}\left(\mathbf{u}\cdot\nabla\right)\mathbf{q}\cdot\mathbf{u}dx=\int_{\omega}\left(\partial_{t}\eta\right)^{2}\xi dx'
  \]
 we will rewrite (in order to linearize) the convective term as follows:
\begin{equation}
\begin{aligned}\int_{0}^{T}\int_{\Omega_{\delta\left(t\right)}}\left(\mathbf{u}\cdot\nabla\right)\mathbf{q}\cdot\mathbf{u}dxdt= & \frac{1}{2}\int_{0}^{T}\int_{\Omega_{\delta\left(t\right)}}\left(\mathbf{u}\cdot\nabla\right)\mathbf{q}\cdot\mathbf{u}dx-\frac{1}{2}\int_{\Omega_{\delta\left(t\right)}}\left(\mathbf{u}\cdot\nabla\right)\mathbf{u}\cdot\mathbf{q}dx\\
 & +\frac{1}{2}\int_{\omega}\left(\partial_{t}\eta\right)^{2}\xi dx'dt.
\end{aligned}
\end{equation}


We now use the above to define {\em decoupled time-periodic weak solutions}. For that we use the regularizing operators from Lemma~\ref{lm:regularizers}. Within this section we will suppress the $\varepsilon$ from the notations and will denote $\mathcal{R}\delta:=\mathcal{R}_{\varepsilon}\delta$. At this level we assume $\varepsilon$ is sufficiently small and fixed. 
We introduce the following:

\begin{definition} Given  $\left(\mathbf{v},\delta\right)\in L^{\infty}\left(I\times\omega\right)\times L^{2}\left(I\times\mathbb{R}^{3}\right)$ with $\delta (0,\cdot)=\delta(T,\cdot)$ and  $\left\Vert \delta \right\Vert _{L^{\infty}\left(I\times\omega\right)}<\kappa$,
we say a couple $(\mathbf{u},\eta)\in V_{S, \text{per}}^{\mathcal{R}\delta}$ is   a  time-periodic  weak solution to the \emph{decoupled}, \emph{linear} and  \emph{ regularized} problem  if:\footnote{Note that the  trace operator $\text{tr}:W^{1,2}\left(\Omega_{\mathcal{R}\delta\left(t\right)}\right)\mapsto W^{1/2,2}\left(\omega\right)\hookrightarrow L^{2}\left(\omega\right)$ is well defined since by  construction  the operator $\mathcal{R}_{\varepsilon}$ enjoys the estimate $\left\Vert \mathcal{R}\delta\right\Vert _{L^{\infty}\left(I\times\omega\right)}\le\left\Vert \delta\right\Vert _{L^{\infty}\left(I\times\omega\right)}<\kappa$.
}
\begin{equation}\label{eq:dec-reg-per-wf}
\begin{aligned}\int_{0}^{T}\int_{\Omega_{\mathcal{R}\delta\left(t\right)}}-\mathbf{u}\cdot\partial_{t}\mathbf{q}+\nabla\mathbf{u}:\nabla\mathbf{q}+\frac{1}{2}\left(\mathbf{u}\cdot\nabla\right)\mathbf{u}\cdot\mathbf{q}-\frac{1}{2}\left(\mathcal{R}\mathbf{v}\cdot\nabla\right)\mathbf{q}\cdot\mathbf{u}dxdt & +\\
\int_{0}^{T}\int_{\omega}-\frac{1}{2}\partial_{t}\eta\partial_{t}\left(\mathcal{R}\delta\right)\xi+\partial_{t}\eta\partial_{t}\xi dx'dt+\int_{0}^{T}\left\langle K^{\prime}\left(\eta\right),\xi\right\rangle dt & =\\
\int_{0}^{T}\int_{\Omega_{\mathcal{R}\delta\left(t\right)}}\mathbf{f}\cdot\mathbf{q}dx+\int_{\omega}g\xi dx'dt\quad\forall(\mathbf{q},\xi)\in V_{T,\text{per}}^{\mathcal{R}\delta}
\end{aligned}
\end{equation}
\end{definition}

The main result of this subsection is the following:
\begin{proposition}\label{prop:existence-dec-reg-per}
Assume (in addition to the previous discussion) that 
\[
\left(\delta,\mathbf{v}\right)\in\left(L^{\infty}\left(I;H_{0}^{2}\left(\omega\right)\right)\cap W^{1,\infty}\left(I;L^{2}\left(\omega\right)\right)\right)\times L^{2}\left(I;W^{1,2}\left(I;\mathbb{R}^{3}\right)\right)
\] 

with 
\begin{equation}\label{eqn:M1+M2}
\left\Vert \delta\right\Vert _{L^{\infty}\left(I;H_{0}^{2}\left(\omega\right)\right)}+\left\Vert \delta\right\Vert _{W^{1,\infty}\left(I;L^{2}\left(\omega\right)\right)}\le M_{1} \  \text{and} \ \left\Vert \mathbf{v}\right\Vert _{L^{2}\left(I;W^{1,2}\left(I;\mathbb{R}^{3}\right)\right)}\le M_{2}
\end{equation}
where $M_1$ is sufficiently small so that ensures $\left\Vert \delta\right\Vert _{L^{\infty}\left(I\times\omega\right)}<\kappa$ (see \eqref{eqn:M_1} for the precise choice).  Let $m\in \mathbb{R}$.
Then there exists a time-periodic weak solution $(\mathbf{u},\eta)$ for the decoupled and regularized problem with data $(\mathbf{v},\delta)$ such that 
\begin{equation}
\int_{\omega}\eta\left(t,\cdot\right)dx'=m \text{ for all }t\in I.
\end{equation}
 and which satisfies  \begin{equation}\label{eqn:periodic-estimate}
\int_{0}^{T}\int_{\Omega_{\mathcal{R}\delta\left(t\right)}}\left|\nabla \mathbf{u}\right|^{2}dxdt\le\int_{0}^{T}\int_{\Omega_{\mathcal{R}\delta\left(t\right)}}\mathbf{f}\cdot \mathbf{u}dx+\int_{0}^{T}\int_{\omega}g\partial_{t}\eta dx'dt.
\end{equation}
and, after denoting \[
E\left(t\right):=\int_{\mathcal{R}\delta\left(t\right)}\left|\mathbf{u}\right|^{2}dx+\int_{\omega}\left|\partial_{t}\eta\right|^{2}dx'+K\left(\eta(t)\right),\ t\in\left[0,T\right]
\]
that \footnote{Recall the definition of $C(\mathbf{f},g)$ from \eqref{eqn:C(f,g)}.}
\begin{equation}\label{eqn:energy-dec-estimate}
\sup_{t\in\left[0,T\right]}E(t)+\int_{0}^{T}\int_{\mathcal{R}\delta\left(t\right)}\left|\nabla\mathbf{u}\right|^{2}dxdt\le c\left(\Omega,\kappa,c_{0},T\right)\cdot\left(C\left(\mathbf{f},g\right)^{2}+C\left(\mathbf{f},g\right)+m^{2}\right)
\end{equation}
The proof of the proposition is split into several parts.
\end{proposition}

\subsection{The construction of a decoupled, regularized, periodic solution}
We consider a sequence $(\mathbf{f}_n,, g_n)$ of smooth functions  such that 
\[
\left(\mathbf{f}_{n},g_{n}\right)\xrightarrow{n\to\infty}\left(\mathbf{f},g\right)\ \text{in}\ L_{\text{loc}}^{2}\left(I\times\mathbb{R}^{3}; \mathbb{R}^{3}\right)\times L_{\text{loc}}^{2}\left(\mathbb{R}\times\omega; \mathbb{R}\right) \]
and \begin{equation}
\int_{0}^{T}\int_{\mathbb{R}^{3}}\left|\mathbf{f}_{n}\right|^{2}dxdt+\int_{0}^{T}\int_{\omega}\left|g_{n}\right|^{2}dx'dt\le\int_{0}^{T}\int_{\mathbb{R}^{3}}\left|\mathbf{f}\right|^{2}dxdt+\int_{0}^{T}\int_{\omega}\left|g\right|^{2}dx'dt=:C(\mathbf{f},g).
\end{equation}

Next  let   $(\hat{Y}_{k})_{k \ge 1}$ be a basis of $\left\{ Y\in H_{0}^{2}\left(\omega\right):\ \int_{\omega}Ydx'=0\right\} $ which is orthogonal in $L^{2} (\omega)$.  
Proceeding in the same manner as in \cite{Gr05} 
we  extend each $\hat{Y}_k$ to a divergence-free function $\mathbf{U}_k : \Omega \mapsto \mathbb{R}^{3}$ by solving the Stokes problem \[\begin{cases}
-\Delta \mathbf{U}_{k}+\nabla P_{k}=0 & \text{in}\ \Omega\\
\text{div} \ \mathbf{U}_{k}=0 & \text{in}\ \Omega\\
\mathbf{U}_{k}=\hat{Y}_{k}\mathbf{e}_3& \text{on}\ \partial\Omega
\end{cases}\]
Then we use the Piola transform from Lemma~ \ref{lm:Piola} and define 
\begin{equation}
\label{eq:Y}
\left(\mathbf{Y}_{k}\left(t,\cdot\right),Y_{k}\left(t,\cdot\right)\right)=\left(\mathcal{J}_{\mathcal{R\delta}\left(t\right)}\mathbf{U}_{k},\hat{Y}_{k}\right)\in V_{T,\text{per}}^{\mathcal{R}\delta}
\end{equation}
We consider also $(\hat{Z}_{k})_{k \ge 1}$ an $L^{2}$-orthonormal basis of the space $H^{1}_{0, \text{div}}$ which we extend in the same manner to \begin{equation}
\left(\mathbf{Z}_{k}\left(t,\cdot\right),Z_{k}\left(t,\cdot\right)\right)=\left(\mathcal{J}_{\mathcal{R\delta}\left(t\right)}Z_{k},0\right)\in V_{T,\text{per}}^{\mathcal{R}\delta}.
\end{equation}
Let us finally define \begin{equation}\label{eqn:defn-Xk} 
\left(\mathbf{X}_{k}\left(t,\cdot\right),X_{k}\left(t,\cdot\right)\right):=\begin{cases}
\left(\mathbf{Y}_{(k+1)/2}\left(t,\cdot\right),Y_{(k+1)/2}\left(t,\cdot\right)\right) & k\ \text{odd}\\
\left(\mathbf{Z}_{k/2}\left(t,\cdot\right),0\right) & k\ \text{even}
\end{cases}\quad\text{for all }k\in\mathbb{N}
\end{equation}
\begin{remark}\label{remark:density-in-testfunctions}
With reasons  very similar to the ones from \cite[p. 234]{LR14} it can be proved that   the space  \[\text{span}\left\{ \varphi\mathbf{X}_{k},\varphi X_{k}:\varphi\in C\left[0,T\right]\cap C^{1}\left(0,T\right),\ \varphi\left(0\right)=\varphi\left(T\right)\right\} \]
 is dense in the space of test functions $V_{T, \text{per}} ^{\mathcal{R}\delta}$. 
 \end{remark} 
We now make the following ansatz:
\begin{equation}\label{eqn:ansatz-u_n-eta_n}
\begin{split}
\eta_{n}\left(t,x'\right) & =\sum_{k=1}^{n}b_{n}^{k}\left(t\right)X_{k}(x')+m\Psi, \ t\in\left[0,T\right],x'\in\omega\\
\mathbf{u}_{n}\left(t,x\right) & =\sum_{k=1}^{n}a_{n}^{k}\left(t\right)\mathbf{X}_{k}\left(t,x\right),\ t\in\left[0,T\right],\ x\in\Omega_{\mathcal{R}\delta\left(t\right)}.
\end{split}
\end{equation}
and we impose the condition $\text{tr}_{\delta}\mathbf{u}_{n}\circ \phi_\delta=\partial_{t}\eta_{n}\mathbf{e}_3$, which means here precisely
\begin{equation}\label{eqn:tr(un)=deltatn}
a_{n}^{k}\left(t\right)=\left(b_{n}^{k}\right)^{\prime}\left(t\right),\, t\in\left[0,T\right],\, k\in\{1,...,n\},
\end{equation} 
by \eqref{eq:Y}.
Due to the trace operator (see Lemma~\ref{lm:trace}) we get 
\begin{equation}\label{eqn:trace(eta_n)}
 \int_{\omega}\left|\partial_{t}\eta_n\right|^{2}dx'\le c_{1}\left(\Omega,k\right)\int_{\Omega_{\mathcal{R\delta}\left(t\right)}}\left|\mathbf{u}_{n}\right|^{2}dx \text{ for all }t\in\left[0,T\right].
\end{equation}
We now  seek for 
\[
\mathbf{b}_{n}:=\left({b}_{n}^{k}\right)_{k=1}^{n}:\left[0,T\right]\mapsto\mathbb{R}^{n}\]
such that the following equation is satisfied:

\begin{align}\label{eqn:xk}
\begin{aligned}
&\int_{\Omega_{\mathcal{R}\delta\left(t\right)}}\partial_{t}\mathbf{u}_{n}\cdot\mathbf{X}_{k}dx+\frac{1}{2}\int_{\omega}\partial_{t}\eta_{n}\partial_{t}\mathcal{R}\delta Y_{k}dx'+\int_{\Omega_{\mathcal{R}\delta\left(t\right)}}\nabla \mathbf{u}_{n}\cdot\nabla\mathbf{X}_{k}dx
\\
&\quad +\frac{1}{2}\int_{\Omega_{\mathcal{R}\delta\left(t\right)}}\left(\mathcal{R}\mathbf{v}\cdot\nabla\right)\mathbf{u}_{n}\cdot\mathbf{X}_{k}dx-\frac{1}{2}\int_{\Omega_{\mathcal{R}\delta\left(t\right)}}\left(\mathcal{R}\mathbf{v}\cdot\nabla\right)\mathbf{X}_{k}\cdot \mathbf{u}_{n}dx+\int_{\omega}\partial_{tt}\eta_{n}X_{k}dx'+\left\langle K^{\prime}\left(\eta_{n}\right),X_{k}\right\rangle 
\\ 
&=\int_{\Omega_{\mathcal{R}\delta\left(t\right)}}\mathbf{f}_{n}\cdot\mathbf{X}_{k}dx+\int_{\omega}g_{n}X_{k}dx',\text{ for all }t\in [0,T] \text{ and for all } 1\le k \le n
\end{aligned}
\end{align}
Please note  that   \eqref{eqn:tr(un)=deltatn} and \eqref{eqn:xk} read as a second order  linear ordinary differential equation with the unknown $\mathbf{b}_n$. Now we prove that for any given initial values $\mathbf{b}_{n}\left(0\right),\mathbf{b}_{n}^{\prime}\left(0\right)$ this differential equation has a unique solution: the coefficient of $\left(\mathbf{b}_{n}\right)^{\prime\prime}$ is given by the mass matrix
\[
M\left(t\right):=\left(\int_{\Omega_{\delta\left(t\right)}}\mathbf{X}_{i}\cdot\mathbf{X}_{j}dx\right)_{1\le i,j\le n}+I_{n}=:\tilde{M}\left(t \right)+I_n.
\]
Since for any $\xi \in \mathbb{R}^{n}$ we have that 
\begin{equation}
\begin{aligned}M\left(t\right)\xi\cdot\xi & =\tilde{M}\left(t\right)\xi\cdot\xi+\text{diag}\left(1,0,1,0,\ldots\right)\xi\cdot\xi\\
 & \ge\sum_{i,j=1}^{n}\xi_{i}\xi_{j}\int_{\Omega_{\delta\left(t\right)}}\mathbf{X}_{i}\cdot\mathbf{X}_{j}dx\\
 & \ge\int_{\Omega_{\delta\left(t\right)}}\left(\sum_{i=1}^{n}\xi_{i}\mathbf{X}_{i}\right)\cdot\left(\sum_{j=1}^{n}\xi_{j}\mathbf{X}_{j}\right)dx\\
 & \ge0
\end{aligned}
\end{equation}
it follows that $M\xi=0$ implies that $\sum_{i}\xi_i \mathbf{X}_{i}=0$ and this in turn  yields (after applying the inverse of the Piola transform $\mathcal{J}_{\mathcal{R}\delta(t)}$) that $\xi=\mathbf{0}$ due to linear independence of the vectors from \eqref{eqn:defn-Xk} 
%
%
Hence it follows that the vector field $\mathbf{b}_n$ exists in an interval $[0,T_0]$ for some $T_0 >0$. 
We introduce the energy for the discretized system: 
\begin{equation}\label{eqn:def-E(t)}
E_{n}\left(t\right):=\int_{\Omega_{\mathcal{R}\delta\left(t\right)}}\left|\mathbf{u}_{n}\left(t\right)\right|^{2}dx+\int_{\omega}\left|\partial_{t}\eta_{n}\left(t\right)\right|^{2}dx'+K\left(\eta_{n}\left(t\right)\right),\ t\in\left[0,T\right].
\end{equation}
By multiplying \eqref{eqn:xk} with $(a_{n} ^{k}$ and summing over $k=1,n$ we do indeed get 
 \begin{equation}\label{eqn:energy-1}
\frac{d}{dt}E_{n}\left(t\right)+\int_{\Omega_{\mathcal{R}\delta\left(t\right)}}\left|\nabla \mathbf{u}_{n}\right|^{2}dx=\int_{\Omega_{\mathcal{R}\delta\left(t\right)}}\mathbf{f}_{n}\cdot \mathbf{u}_{n}dx+\int_{\omega}g_{n}\eta_{n}dx' \text{ for all }t\in\left[0,T\right].
\end{equation}
In particular this allows to extend $\mathbf{b}_n$ to the whole  interval  $[0,T]$.

\subsection{The key estimate}

The following proposition is the main effort to obtain the existence of a fixed point.
\begin{proposition}\label{prop: invariant-ball} 
Suppose that \[E_{n} (T) \ge E_{n} (0)\]
(or,  more general, that \eqref{eqn:bad-case} holds). 
Then
there exists a constant    
\[R=c\left(\Omega,\kappa,c_{0},T\right)\cdot\left(C\left(\mathbf{f},g\right)^{2}+C\left(\mathbf{f},g\right)+m^{2}\right)>0\]
for which
\[
\ E_{n}\left(T\right)\le R.\]
In particular \[E_{n} \left(0 \right) \le E_{n} \left(T \right) \le R.\]
\end{proposition}
\begin{proof}
The assumption $E_{n} (T) \ge E_{n} (0)$  yields 
\begin{equation}\label{eqn:bad-case}
\int_{0}^{T}\int_{\Omega_{\mathcal{R}\delta\left(t\right)}}\mathbf{f}_{n}\cdot\mathbf{u}_{n}dxdt+\int_{0}^{T}\int_{\omega}g_{n}\partial_{t}\eta_{n}dx'dt\ge\int_{0}^{T}\int_{\Omega_{\mathcal{R}\delta\left(t\right)}}\left|\nabla\mathbf{u}_{n}\right|^{2}dxdt.
\end{equation}
From  \eqref{eqn:bad-case}, Young's inequality,  Poincare's inequality and the trace operator (see Lemma~\ref{lm:trace}) we get
 \begin{equation}\label{eqn:consequence-claim}
\int_{0}^{T}\int_{\Omega_{\mathcal{R}\delta\left(t\right)}}\left|\nabla\mathbf{u}_{n}\right|^{2}+\left|\mathbf{u}_{n}\right|^{2}dx+\int_{\omega}\left|\partial_{t}\eta_{n}\right|^{2}dx'dt\le c_{2}\left(\Omega,\kappa\right)\left(\int_{0}^{T}\int_{\Omega_{\mathcal{R}\delta\left(t\right)}}\left|\mathbf{f}\right|^{2}dxdt+\int_{\omega}\left|g\right|^{2}dxdt\right).
\end{equation}
Recall the relation \eqref{eqn:xk}, and note that we can, for a function $\phi\in C^{\infty} (0,T)$ with $\phi \le 1$, $\phi \equiv 1$ on $[T/4,T]$ and $\phi^{\prime} \le \frac{c}{T}$, 
use the test function $
\phi^{3}\left(\mathbf{u}_{n},\partial_{t}\eta_{n}\right)
$
and  obtain

\begin{equation}\label{eqn:En-first-semicircle}
\begin{split}
\phi^{3}E_{n}\left(t\right)\le & c\int_{0}^{T}\int_{\Omega_{\mathcal{R}\delta\left(t\right)}}\phi^{2}\phi^{\prime}\left|\mathbf{u}_{n}\right|^{2}+\phi^{3}\mathbf{f}_{n}\cdot\mathbf{u}_{n}dxdt+\\
 & c\int_{0}^{T}\int_{\omega}\phi^{3}g_{n}\partial_{t}\eta_{n}+\phi^{2}\left(\partial_{t}\eta_{n}\right)^{2}dx'dt+c\int_{0}^{T}\phi^{2}\phi^{\prime}K\left(\eta_{n}\right)dt\\
\le & c\frac{c\left(\mathbf{f},g\right)}{T}+c\fint_{0}^{T/4}\phi^{2}K\left(\eta_{n}\right)dt
\end{split}
\end{equation}
We aim to estimate uniformly the term $\fint_{0}^{T/4}E\left(t\right)dt$.

To this end,  let us consider a function $\psi\in C_{0}^{\infty}\left(0,T\right)$ such that $\psi = \phi$ on $[0, 3T/4]$.  Note that we may assume that $\psi ^{\prime} \le \frac{c}{T}$. 
By  multiplying the equation~\eqref{eqn:xk} with $\psi^{2}b_{n}^{k}$ we are allowed to use the test function
 \begin{equation}\label{eqn:B_n}
\psi^{2}\left(\mathbf{B}_{n},\mathcal{M}\left(\eta_{n}\right)\right):=\psi^{2}\left(\sum_{k=1}^{n}b_{n}^{k}\mathbf{X}_{k},\sum_{k=1}^{n}b_{n}^{k}X_{k}\right)\in V_{T,}^{\mathcal{R}\delta},
 \end{equation}
this is the discrete represent of the test-function related to ''$\eta-m$''.
We obtain
\begin{equation}
\begin{aligned}\int_{0}^{T}\psi^{2}K\left(\eta_{n}\right)dt\le & \int_{0}^{T}\int_{\Omega_{\mathcal{R}\delta\left(t\right)}}\psi^{2}\left(\left|\nabla\mathbf{u}_{n}:\nabla\mathbf{B}_{n}\right|+\left|\nabla\mathbf{B}_{n}:\mathcal{R}\mathbf{v}\otimes\mathbf{u}_{n}\right|+\left|\nabla\mathbf{u}_{n}:\mathbf{B}_{n}\otimes\mathcal{R}\mathbf{v}\right|\right)dxdt+\\
 & \int_{0}^{T}\int_{\Omega_{\mathcal{R}\delta\left(t\right)}}\left|\mathbf{u}_{n}\cdot\partial_{t}\left(\psi^{2}\mathbf{B}_{n}\right)\right|dxdt+\int_{0}^{T}\psi^{2}\int_{\omega}\left|\eta_{n}\partial_{t}\eta_{n}\partial_{t}\mathcal{R}\delta\right|+\left|\partial_{t}\eta_{n}\partial_{t}\left(\eta_{n}\phi\right)\right|dx'dt+\\
 & +\int_{0}^{T}\psi^{2}\left\langle K^{\prime}\left(\eta_{n}\right),m\Psi\right\rangle dt+c\left(\mathbf{f},g\right)\\
 & =:\sum_{k=1}^{7}I_{k}+C\left(\mathbf{f},g\right)
\end{aligned}
\end{equation}
Now we need some properties of $\mathbf{B}_n$ (from \eqref{eqn:B_n}),  which follow from \eqref{eqn:bad-case} and the definition of $\mathbf{X}_k$ and $X_k$:
 \begin{equation}\label{eqn:properties-B_n}
\int_{0}^{T}\int_{\Omega_{\mathcal{R}\delta\left(t\right)}}\left|\partial_{t}\mathbf{B}_{n}\right|^{2}+\left|\mathbf{B}_{n}\right|^{2}+\left|\nabla \mathbf{B}\right|^{2}dx\le C\left(\Omega,\kappa,c_{0}\right)\int_{0}^{T}K\left(\eta_{n}\left(t,\cdot\right)\right)dt+C(\mathbf{f},g).
\end{equation}
Let $\theta >0$; we use Young's inequality and \eqref{eqn:bad-case}, \eqref{eqn:properties-B_n} to get that 
\begin{equation}\label{eqn:I1+I2}
I_{1}\le\theta\int_{0}^{T}\int_{0}^{T}\psi^{2}K\left(\eta_{n}\left(t\right)\right)dt+\frac{c(C(\mathbf{f},g)))}{\theta}
\end{equation}
and for $\theta$ sufficiently small we can absorb the $K(\eta_n)$ term into the left-hand side.
Next,

\begin{equation}\label{eqn:I3+I4+I5}
\begin{aligned}
\sum_{k=2}^{6}I_{k} & \le\left\Vert \psi^{2}\nabla\mathbf{B}_{n}\right\Vert _{L_{t}^{\infty}L_{x}^{3/2}}\left\Vert \mathcal{R}\mathbf{v}\right\Vert _{L_{t}^{2}L_{x}^{6}}\left\Vert \mathbf{u}_{n}\right\Vert _{L_{t}^{2}L_{x}^{6}}+\left\Vert \psi^{2}\nabla\mathbf{B}_{n}\right\Vert _{L_{t}^{\infty}L_{x}^{3}}\left\Vert \mathcal{R}\mathbf{v}\right\Vert _{L_{t}^{2}L_{x}^{6}}\left\Vert \nabla\mathbf{u}_{n}\right\Vert _{L_{t}^{2}L_{x}^{2}}
\\
&\quad +
 \left\Vert \partial_{t}\mathcal{R}\mathbf{\delta}\right\Vert _{L_{t}^{2}L_{x}^{2}}\left\Vert \partial_{t}\eta_{n}\right\Vert _{L_{t}^{2}L_{x}^{2}}\left\Vert \psi^{2}\eta_{n}\right\Vert _{L_{t}^{\infty}L_{x}^{\infty}}+cC(\mathbf{f},g)
 \\
 & \le\theta\sup_{t\in\left[0,T\right]}\psi^{2}K\left(\eta_{n}\left(t\right)\right)+\frac{c((C(\mathbf{f},g))+(C(\mathbf{f},g))^2)}{\theta}\\
 & \le\theta\sup_{t\in\left[0,T\right]}\psi^{2}E_{n}\left(t\right)+\frac{c((C(\mathbf{f},g))+(C(\mathbf{f},g))^2)}{\theta}
\end{aligned}
\end{equation}
and 
\begin{equation}\label{eqn:I7}
\begin{aligned}
I_{7}\le\theta\left\Vert \eta\right\Vert _{L_{t}^{2}H_{x}^{2}}^{2}+\frac{c}{\theta}m^{2}
\end{aligned}
\end{equation}
From \eqref{eqn:En-first-semicircle}, \eqref{eqn:I1+I2},   \eqref{eqn:I3+I4+I5} \eqref{eqn:I7} and the fact that $\psi \le \phi$  we conclude that
 \begin{equation}
\sup_{t\in\left[0,T\right]}\phi^{2}E_{n}\left(t\right)\le c\left(\Omega,\kappa,c_{0},T\right)\cdot\left(C\left(\mathbf{f},g\right)^{2}+C\left(\mathbf{f},g\right)+m^{2}\right)
\end{equation}
and in particular 
\begin{equation}
E_{n}(T) \le c\left(\Omega,\kappa,c_{0},T\right)\cdot\left(C\left(\mathbf{f},g\right)^{2}+C\left(\mathbf{f},g\right)+m^{2}\right).
\end{equation}
This now proves  Proposition~\ref{prop: invariant-ball} by considering 
\begin{equation}\label{eqn:R}
R:=c\left(\Omega,\kappa,c_{0},T\right)\cdot\left(C\left(\mathbf{f},g\right)^{2}+C\left(\mathbf{f},g\right)+m^{2}\right).
\end{equation}


%
%

%

\subsection{End of the proof of Proposition~\ref{prop:existence-dec-reg-per}} 
We aim to prove that the following mapping possesses a fixed-point:
\begin{equation}\label{eqn:F_n}
F_{n}:\mathbb{R}^{2\times n}\mapsto\mathbb{R}^{2\times n},F_{n}:\left(\mathbf{b}_{n}\left(0\right),\mathbf{b}_{n}^{\prime}\left(0\right)\right)\mapsto\left(\mathbf{b}_{n}\left(T\right),\mathbf{b}_{n}^{\prime}\left(T\right)\right).
\end{equation} This mapping  associates to any set of initial data the solution $\mathbf{b}_n$ of \eqref{eqn:xk} evaluated at $t=T$. 
We will apply Sch\"{a}ffer's Fixed Point Theorem.
 The mapping $F_n$ is \emph{continuous} since a
  solution of the above linear ODE depends continuously on the initial data.  The \emph{compactness} is automatically ensured because we work in a finite dimensional Banach space $\mathbb{R}^{2\times n}$.

%
In order to use Sch\"{a}ffer's fixed-point Theorem~\ref{thm:Schaeffer} it suffices to prove that the set
\begin{equation}\label{eqn:set-Schauder}
S:=\left\{ \left(\mathbf{b}_{n}\left(0\right),\mathbf{b}_{n}^{\prime}\left(0\right)\right):\left(\mathbf{b}_{n}\left(0\right),\mathbf{b}_{n}^{\prime}\left(0\right)\right)=\lambda\left(\mathbf{b}_{n}\left(T\right),\mathbf{b}_{n}^{\prime}\left(T\right)\right)\ \text{for some}\ \lambda\in\left[0,1\right]\right\} 
\end{equation}
is uniformly bounded with respect to $\lambda$.  Let us consider $\left(b_{0},b_{1}\right)\in S$.  We can exclude the trivial case $\lambda=0$ and assume $0<\lambda\le 1$.  Then it follows that 
\[
E_{n}\left(T\right)=\frac{1}{\lambda^{2}}E_{n}\left(0\right)\ge E_{n}\left(0\right).
\]
But now Proposition~\ref{prop: invariant-ball} provides a uniform bound on $S$ and a fixed-point is obtained. 
The bounds on $S$ are in particular bounds on the fixed-point and imply
\begin{equation}\label{eqn:formaln3}
\int_{0}^{T}\int_{\Omega_{\mathcal{R}\delta\left(t\right)}}\left|\mathbf{u}_{n}\right|^{2}+\left|\nabla\mathbf{u}_{n}\right|^{2}dxdt+\int_{\omega}\left|\partial_{t}\eta_{n}\right|^{2}dx'dt\le c\left(\Omega,\kappa\right)C\left(\mathbf{f},g\right).
\end{equation}
Further the energy equality is satisfied for the fixed-point value:
\begin{equation}\label{formaln1}
\int_{0}^{T}\int_{\Omega_{\mathcal{R}\delta\left(t\right)}}\left|\nabla\mathbf{u}_{n}\right|^{2}dxdt=\int_{0}^{T}\int_{\Omega_{\mathcal{R}\delta\left(t\right)}}\mathbf{f}_{n}\cdot\mathbf{u}_{n}dx+\int_{\omega}g_{n}\partial_{t}\eta_{n}dx'dt.
\end{equation}

Hence we are left to pass to the limit with $n\to \infty$.
From \eqref{eqn:M1+M2}
we may deduce that $\left\Vert \delta\right\Vert _{L^{\infty}\left(I\times\omega\right)}\le c\left(\omega\right)M_{1}$ we will choose 
\begin{equation}\label{eqn:M_1}
M_{1}\le\frac{3\kappa}{4c\left(\omega\right)}
\end{equation}
and will assume the data to be small enough such that $R$ is accordingly.
while for $M_2$, since we want the mapping $\mathbf{v}\mapsto \mathbf{u}$  to admit a fixed point,  we choose $M_2$ so that the bounds on $\mathbf{v,u}$ are  uniform w.r.t.  $\varepsilon$ . In accordance with \eqref{eqn:formaln3} we choose
 \begin{equation}\label{eqn:M_2}
M_{2}\le c\left(c_{0},\Omega,\kappa\right)C(\mathbf{f},g).
\end{equation}
 Now let $ t\in [0,T]$ be a fixed time and let us integrate \eqref{eqn:energy-1} on $[0,t]$.  By using  \eqref{eqn:consequence-claim} and \eqref{eqn:R} we get 
\begin{alignat*}{1}
E_{n}\left(t\right) & =E_{n}\left(0\right)+\int_{0}^{t}\int_{\Omega_{\mathcal{R}\delta\left(t\right)}}\mathbf{f}_{n}\cdot \mathbf{u}_{n}dxdt+\int_{0}^{t}\int_{\omega}g_{n}\partial_{t}\eta_{n}dx'ds\\
 & \le R+\int_{0}^{T}\int_{\Omega_{\mathcal{R}\delta\left(t\right)}}\left|\mathbf{f}_{n}\cdot \mathbf{u}_{n}\right|dxdt+\int_{0}^{T}\int_{\omega}\left|g_{n}\partial_{t}\eta_{n}\right|dx'ds\\
 & \le R+C\left(\omega, \kappa \right)C(\mathbf{f},g)
 \end{alignat*}
which then reads as 
\begin{equation}\label{eqn:uniform-periodic-time-estimate }
\sup_{t\in\left[0,T\right]}E_{n}(t)\le C\left(\Omega,\kappa,c_{0},T\right)\cdot\Big(\left(C\left(\mathbf{f},g\right)\right)^{2}+C\left(\mathbf{f},g\right)+m^{2}\Big).
\end{equation}
Now  from \eqref{eqn:uniform-periodic-time-estimate } it follows that there is $(\mathbf{u},\eta)$ for which 
\begin{equation}\label{eqn:limit-subsequence}
\begin{split}
\eta_{n}\to\eta & \text{\ weakly * \ in}\ L^{\infty}\left(0,T;H_{0}^{2}\left(\omega\right)\right)\\
\partial_{t}\eta_{n}\to\partial_{t}\eta & \text{\ weakly * \ in}\ L^{\infty}\left(0,T;L^{2}\left(\omega\right)\right)\\
\mathbf{u}_{n}\to \mathbf{u} & \text{\ weakly * \ in}\ V_{F,\text{per}}^{\mathcal{R}\delta}.
\end{split}
\end{equation}
By weak lower semicontinuity we have 
\begin{equation}\label{eqn:diff-small}
\int_{0}^{T}\int_{\Omega_{\mathcal{R}\delta\left(t\right)}}\left|\nabla \mathbf{u}\right|^{2}dxdt\le c(c_0, \kappa,\Omega) \int_{0}^{T}\int_{\Omega_{\mathcal{R}\delta\left(t\right)}}\mathbf{f}\cdot \mathbf{u}dx+\int_{0}^{T}\int_{\omega}g\partial_{t}\eta dx'dt.
\end{equation}
Using \eqref{eqn:limit-subsequence} and the  Remark~\ref{remark: Holder-cont} it follows that \[ \eta (0,\cdot) =\eta(T,\cdot).\]
Note that from  \eqref{eqn:limit-subsequence}  and $\text{tr}_{\mathcal{R}\delta}\mathbf{u}_{n}\circ \phi_{\mathcal{R}\delta}=\partial_{t}\eta_{n}\ebf$  by letting $n \to \infty$ we obtain $\text{tr}_{\mathcal{R}\delta}\mathbf{u}\circ \phi_{\mathcal{R}\delta}=\partial_{t}\eta \ebf.$
To end the proof,  we multiply \eqref{eqn:xk} by $\varphi\in C^{1}\left(0,T\right)\cap C^{0}[0,T],\ \varphi\left(0\right)=\varphi\left(T\right)$ and integrate by parts on $[0,T]$.  We get 
\begin{align}\label{eqn:density}
\begin{aligned}
&\int_{0}^{T}\int_{\Omega_{\mathcal{R}\delta\left(t\right)}}-\mathbf{u}_{n}\cdot\partial_{t}\left(\varphi\left(t\right)\mathbf{X}_{k}\right)dx+\nabla\mathbf{u}_{n}:\nabla\left(\varphi\left(t\right)\mathbf{X}_{k}\right)dxdt
\\
&\quad  +
\int_{0}^{T}\int_{\Omega_{\mathcal{R}\delta\left(t\right)}}\frac{1}{2}\left(\mathcal{R}\mathbf{v}\cdot\nabla\right)\mathbf{u}_{n}\cdot\left(\varphi\left(t\right)\mathbf{X}_{k}\right)-\frac{1}{2}\left(\mathcal{R}\mathbf{v}\cdot\nabla\right)\left(\varphi\left(t\right)\mathbf{X}_{k}\right)\cdot\mathbf{u}_{n}dxdt 
\\
&\quad  -
\int_{0}^{T}\int_{\omega}\frac{1}{2}\partial_{t}\eta_{n}\left(\varphi\left(t\right)X_{k}\right)\partial_{t}\mathcal{R}\delta\ dx'dt+\partial_{t}\eta_{n}\partial_{t}\left(\varphi\left(t\right)X_{k}\right)dx'+\int_{0}^{T}\left\langle K^{\prime}\left(\eta_{n}\right),X_{k}\right\rangle dt 
\\
& =
\int_{0}^{T}\int_{\Omega_{\mathcal{R}\delta\left(t\right)}}\mathbf{f}_{n}\cdot\left(\varphi\left(t\right)\mathbf{X}_{k}\right)dxdt+\int_{0}^{T}\int_{\omega}g_{n}\varphi\left(t\right)X_{k}dx'dt. 
\end{aligned}
\end{align}
for all $k\in \{1,...,n\}$
We can now take an arbitrary linear combination of the above and let $n\to \infty$ in \eqref{eqn:density}; note that the weak convergences from \eqref{eqn:limit-subsequence} are enough, since the problem is linear. Recall  Remark~\ref{remark:density-in-testfunctions} and the proof is complete. 
\end{proof}

\subsection{$\varepsilon$-regularized periodic solutions}
 This part follows the already developed existence theory for the respective Cauchy problem~\cite{Gr05,MC13,LR14,MS19}. Indeed, since the a-priori estimates for the here constructed time-periodic system are the same as for the Cauchy problem, the existence of solutions follows by the analysis provided there. For the convenience of the reader we repeat here the main steps that have to be performed in order to produce a weak solution. For the technical details we do however refer to the references given above.  

During this subsection we regain the coupling between fluid and shell, that is we prove that the mapping $(\delta, \mathbf{v}) \mapsto (\eta, \mathbf{u})$ has a fixed point. To this end we use Theorem~\ref{thm: Kakutain} which can be found in the Appendix.  
However, note that the motion of the shell $\eta$ will still depend on $\varepsilon$, since we still need its regularity ensured by the operators introduced in Lemma~\ref{lm:regularizers}. The last step will therefore consist of  letting $\varepsilon \to 0$ however this procedure takes place in the next  section only. 

\begin{definition}\label{defn:eps-reg-soln}
We say that a couple $(\mathbf{u},\eta)=(\mathbf{u}_\varepsilon, \eta_\varepsilon ) \in V_{S, \text{per}}^{\mathcal{R}_{\varepsilon} \eta}$ is a weak time-periodic solution of the $\varepsilon$-regularized problem in the interval $I=(0,T)$ if  $\left\Vert \eta\right\Vert _{L^{\infty}\left(I\times\omega\right)}<\kappa$,   if
\begin{equation}\label{eq:eps-regularized}
\begin{split}
\int_{0}^{T}\int_{\Omega_{\mathcal{R}_{\varepsilon}\eta\left(t\right)}}-\mathbf{u}\cdot\partial_{t}\mathbf{q}+\nabla\mathbf{u}:\nabla\mathbf{q}+\frac{1}{2}\left(\mathcal{R}_{\varepsilon}\mathbf{u}\cdot\nabla\right)\mathbf{u}\cdot\mathbf{q}-\frac{1}{2}\left(\mathcal{R}_{\varepsilon}\mathbf{u}\cdot\nabla\right)\mathbf{q}\cdot\mathbf{u}dxdt & -\\
\int_{0}^{T}\int_{\omega}\frac{1}{2}\partial_{t}\eta\partial_{t}\mathcal{R}_{\varepsilon}\eta\xi+\partial_{t}\eta\partial_{t}\xi dx'+\left\langle K_{\varepsilon}^{\prime}\left(\eta\right),\xi\right\rangle dt & =\\
\int_{0}^{T}\int_{\Omega_{\mathcal{R}_{\varepsilon}\eta\left(t\right)}}\mathbf{f}\cdot\mathbf{q}dx+\int_{\omega}g\xi dx'dt & \forall\left(\mathbf{u},\xi\right)\in V_{T,\text{per}}^{\mathcal{R}_{\varepsilon}\eta}
\end{split}.
\end{equation}
\end{definition}
Then we aim to prove that 
\begin{proposition}\label{prop:eps-reg-per} Let $m\in \mathbb{R}$. 
There exists a constant 
\begin{equation}
C_{0}=C_{0}(c_{0},\Omega,\kappa,T)
\end{equation}
such that if 
 \begin{equation}\label{eqn:cond-f-g}
m^{2}+C\left(\mathbf{f},g\right)^{2}\le C_{0}
\end{equation}
 then  for all sufficiently small $\varepsilon >0$,  there exists at least one  weak time periodic solution $\left(\mathbf{u},\eta\right)=\left(\mathbf{u}_{\varepsilon},\eta_{\varepsilon}\right)$ for the $\varepsilon$-regularized problem such that 
 \begin{equation}\label{eqn:epsilon-zero}
 \int_{\omega}\eta\left(t,\cdot\right)dx'=m\text{ for all }t\in I
 \end{equation}
 Furthermore,  by denoting (as before)  
 \[
E
\left(t\right):=\int_{\Omega_{\mathcal{R}_{\varepsilon}\eta\left(t\right)}}\left|\mathbf{u}\right|^{2}dx+\int_{\omega}\left|\partial_{t}\eta\right|^{2}dx'+K\left(\eta(t)\right),\ t\in\left[0,T\right]
 \]  
it holds that 
 \begin{equation}\label{eq:estimates-eps-reg-soln}
\text{ess}\sup_{t\in I}E
\left(t\right)+\int_{0}^{T}\int_{\Omega_{\mathcal{R}_{\varepsilon}\eta}}\left|\nabla\mathbf{u}\right|^{2}dxdt\le C(\Omega,\kappa,c_{0},T)\cdot\Big(\left(C\left(\mathbf{f},g\right)\right)^{2}+C\left(\mathbf{f},g\right)+m^{2}\Big) 
\end{equation}
and 
\begin{equation}\label{eqn:eps-diff}
\int_{0}^{T}\int_{\Omega_{\mathcal{R}_{\varepsilon}\eta\left(t\right)}}\left|\nabla\mathbf{u}\right|^{2}dxdt\le\int_{0}^{T}\int_{\Omega_{\mathcal{R}_{\varepsilon}\eta\left(t\right)}}\mathbf{f}\cdot\mathbf{u}dxdt+\int_{\omega}g\cdot\partial_{t}\eta dx'dt.
\end{equation}
\end{proposition}
\begin{proof}
Let us  consider the space 
\[Z:= L_{\text{per}}^{2}\left(\mathbb{R};L^{2}\left(\mathbb{R}^{3}\right)\right) \times C_{\text{per}}^{0}\left(\mathbb{R};C^{0}\left(\overline{\omega}\right)\right)\
\] 
 and its  convex subset
  \[
D:=\left\{ \left(\mathbf{v},\delta\right)\in Z:\left\Vert \delta\right\Vert _{L_{t}^{\infty}L_{x}^{\infty}}\le M_{1},\left\Vert \mathbf{v}\right\Vert _{L_{t}^{2}L_{x}^{2}}\le M_{2},\ \int_{\omega}\delta\left(t,\cdot\right)d\equiv m\right\} 
   \]
where $M_1$ and $M_2$ are given by \eqref{eqn:M_1} and \eqref{eqn:M_2}.
We define 
\[
F:Z\mapsto\mathcal{P}(Z),\quad F\left(\mathbf{v},\delta\right):=\left\{ \left(\mathbf{u},\eta\right):\left(\mathbf{u},\eta\right)\in V_{T,\text{per}}^{\mathcal{R}\delta}\text{ is a solution for}\ \eqref{eq:eps-regularized}\right\} .
\] 
Here we extend  for almost every $t\in I$ the function   $\mathbf{u}(t,\cdot)$ to $\mathbb{R}^3$ by zero.
We would like to have 
\[
F : D \subset Z \mapsto \mathcal{P}(D)
\]
and therefore from \eqref{eqn:energy-dec-estimate}, \eqref{eqn:M_1}, \eqref{eqn:M_2} we need to impose the  smallness condition \eqref{eqn:epsilon-zero}.
Further we check the assumptions of  Theorem~\ref{thm: Kakutain}:
 \begin{enumerate}[(i)]
 \item For all $(\mathbf{v},\delta)\in D$ the set $F(\mathbf{v},\delta)$ is \emph{non-empty} (due to Proposition~\ref{prop:existence-dec-reg-per}), 
  \emph{convex} (since the problem is linearized). The fact that $F(D) \subset D$ follows by imposing the  \emph{smallness} condition \eqref{eqn:cond-f-g} on the forces. 
  More precisely, this is due to the estimate~\eqref{eqn:uniform-periodic-time-estimate }.

  \item $F(D)$ is \emph{relatively compact} in $Z$: here we need  to employ a compactness result, for example the one from \cite[Proposition 3.8]{LR14}.  Please observe, that this can indeed be applied since after establishing the  Proposition~\ref{prop: invariant-ball} we are able to select a set of initial data from a uniformly bounded (by $R$, see \eqref{eqn:R}) set.
Note that  for sufficiently small $\varepsilon$ the regularizing operators ensure according to Lemma~\ref{lm:regularizers} that
   \[\mathcal{R}:\left\{ \delta\in C\left(\overline{I}\times\omega\right):\delta\left(0,\cdot\right)=\delta\left(T,\cdot\right)\right\} \mapsto C^{3}\left(\overline{I}\times\omega\right)\hookrightarrow\hookrightarrow C^{2}\left(\overline{I}\times\omega\right).\] 

   \item $F$ has a \emph{closed graph} since if $(\mathbf{v}_{n},\ \delta_{n}) \to (\mathbf{v},\delta)$ and $\left(\mathbf{v}_{n},\delta_{n}\right)\in F\left(\mathbf{v}_{n},\delta_{n}\right)$ with $\left(\mathbf{u}_{n},\eta_{n}\right)\to\left(\mathbf{u},\eta\right)$, then  $(\eta,\mathbf{u}) \in F(\delta,\mathbf{v})$. The latter is an easy exercise especially since the problem is linearized.  To prove that $\eta(0, \cdot)=\eta(T,\cdot)$ it suffices to note that $\eta_{n}(0, \cdot)=\eta_{n} (T, \cdot)$ and we get strong convergence for every time due to Remark~\ref{remark: Holder-cont}.
 \end{enumerate}
Therefore the assumptions of the Theorem~\ref{thm: Kakutain} are fulfilled and we can guarantee the existence of a pair $(\eta, \mathbf{u}) \in D$ for which $(\eta, \mathbf{u}) \in F(\eta,\mathbf{u})$. Once the fixed-point is established, the estimates \eqref{eq:estimates-eps-reg-soln} and  \eqref{eqn:eps-diff} follow  from \eqref{eqn:periodic-estimate} and \eqref{eqn:energy-dec-estimate}.
\end{proof}
\begin{remark}
We would like to point out that the same reasoning would have worked  at the discrete (Galerkin) level,  by replacing $\delta$ with $\delta_n$--a discrete but already  regularized version. We refer to   \cite[Section 4.1]{BS21} where this was successfully performed.
\end{remark} 
 
%


\subsection{$\varepsilon \to 0$ limit and the proof of Theorem~\ref{thm:main}}
\begin{proof}
Our proof follows the same lines as  in  \cite{LR14} and \cite{MS19}.\\
Recalling the definition of $\varepsilon$-regularized periodic solutions from \eqref{eq:eps-regularized}, we know that for all sufficiently small $\varepsilon>0$ it holds that:
\begin{equation}\label{eqn:u_eps}
\begin{aligned}
\int_{0}^{T}\int_{\Omega_{\mathcal{R}_{\varepsilon}\eta_{\varepsilon}\left(t\right)}}-\mathbf{u}_{\varepsilon}\cdot\partial_{t}\mathbf{q}_{\varepsilon}dx+\nabla\mathbf{u}_{\varepsilon}:\nabla\mathbf{q}_{\varepsilon}+\frac{1}{2}\left(\mathcal{R}_{\varepsilon}\mathbf{u}_{\varepsilon}\cdot\nabla\right)\mathbf{u}_{\varepsilon}\cdot\mathbf{q}_{\varepsilon}-\frac{1}{2}\int_{\Omega_{\mathcal{R}_{\varepsilon}\eta\left(t\right)}}\left(\mathcal{R}_{\varepsilon}\mathbf{u}_{\varepsilon}\cdot\nabla\right)\mathbf{q}\cdot\mathbf{u}_{\varepsilon}dxdt+\\
\\
\int_{0}^{T}\int_{\omega}-\partial_{t}\eta_{\varepsilon}\partial_{t}\xi_{\varepsilon}dx'+\left\langle K_{\varepsilon}^{\prime}\left(\eta_{\varepsilon}\right),\xi_{\varepsilon}\right\rangle dt=\int_{0}^{T}\int_{\Omega_{\mathcal{R}_{\varepsilon}\eta\left(t\right)}}\mathbf{f}\cdot\mathbf{q}_{\varepsilon}dx+\int_{\omega}g\xi_{\varepsilon}dx'dt\quad\forall\left(\mathbf{q}_{\varepsilon},\xi_{\varepsilon}\right)\in V_{T,\text{per}}^{\mathcal{R}_{\varepsilon}\eta_{\varepsilon}}
\end{aligned}
\end{equation}
We  obtain the existence of a pair $(\mathbf{u},\eta,)$ for which 
\begin{equation}
\begin{aligned}
\eta_{\varepsilon},\mathcal{R}_{\varepsilon}\eta_{\varepsilon}\to\eta & \quad \text{uniformly and weakly}^{*}\ \text{in}\ L^{\infty}\left(I;H_{0}^{2}\left(\omega\right)\right)\\
\partial_{t}\eta_{\varepsilon},\mathcal{R}_{\varepsilon}\partial_{t}\eta_{\varepsilon}\to\partial_{t}\eta &\quad \text{weakly}^{*}\ \text{in}\ L^{\infty}\left(I;L^{2}\left(\omega\right)\right)\\
\mathbf{u}_{\varepsilon}\to\mathbf{u} & \quad \text{weakly}^{*}\ \text{in}\ L^{\infty}\left(I;L^{2}\left(\mathbb{R}^{3}\right)\right)\\
\nabla\mathbf{u}_{\varepsilon}\to\nabla\mathbf{u} & \quad \text{weakly in}\ L^{2}\left(I;L^{2}\left(\mathbb{R}^{3}\right)\right)
\end{aligned}
\end{equation}
for some subsequences which we choose not to relabel in order to keep the notation light.
Note that this suffices  to ensure that $\eta(0, \cdot)=\eta(T,\cdot)$ in $\omega$ and $\left\Vert \eta\right\Vert _{L^{\infty}\left(I\times\omega\right)}<\kappa$.
This time however  our problem is \emph{nonlinear} and therefore strong convergence is needed in order to pass to the limit in \eqref{eqn:u_eps}.  Following now \cite[Section 4]{Gr05} or \cite[Lemma 6.3]{MS19} we get that 
  \begin{equation} 
\left(\mathbf{u}_{\varepsilon},\partial_{t}\eta_{\varepsilon}\right)\to\left(\mathbf{u},\partial_{t}\eta\right)\quad\text{in}\ L^{2}\left(I;L^{2}\left(\mathbb{R}^{3}\right)\right)\times L^{2}\left(I;L^{2}\left(\omega\right)\right).
\end{equation}
In particular, (by interpolating $L^{4}$ between $L^{2}$ and $L^{5}$) one obtains that \begin{equation}
  \mathbf{u}_{\varepsilon},  \mathcal{R}_{\varepsilon} \mathbf{u}_{\varepsilon}\to \mathbf{u}\quad\text{in}\ L^{2}\left(I;L^{4}\left(\mathbb{R}^{3}\right)\right).
\end{equation}

This now proves Theorem~\ref{thm:main}.
\end{proof}
\section*{Acknowledgments}
C. M. and S. S. thank the support of the Primus research programme PRIMUS/19/SCI/01  and of the program GJ19-11707Y of the Czech National Grant Agency (GA\v{C}R).  They also thank the SVV program SVV-2020-260583 of the Faculty of Mathematics and Physics of  Charles University.

Moreover, S.S. thanks for the support of  the University Centre UNCE/SCI/023 of Charles University.

\bibliographystyle{apalike}

\section{Appendix}\label{section:appendix}
\subsection{Reynolds' transport theorem}
We  recall the classical  Reynolds' Transport Theorem, which in our context (see Subsection~\ref{subsection:domain}) reads as follows
\begin{theorem}\label{thm:Reynolds}
It holds that
 \[\frac{d}{dt}\int_{\Omega_{\eta\left(t\right)}}g\left(t,x\right)dx=\int_{\Omega_{\eta\left(t\right)}}\partial_{t}g\left(t,x\right)dx+\int_{\partial\Omega_{\eta\left(t\right)}}\partial_{t}\eta\circ\psi_{\eta}^{-1}\mathbf{e}_3\cdot \nu_{\eta} gdx'\] whenever all the terms that appear are well-defined.
\end{theorem}
 
  The latter result will be used intensively during this work.
\subsection{Poincar\'{e}'s inequality}
\begin{theorem}\cite[Theorem 13.15]{Le17} \label{thm:poincare} If $\Omega \subset \mathbb{R}^{N}$ is an open set that lies between two hyperplanes which are at distance $d$, $m \in \mathbb{N}$ and $1 \le p < \infty$, then there is a constant $c=c(m,N,p)>0$ such that \[\int_{\Omega}\left|\nabla^{k}\mathbf{u}\right|^{p}dx\le c\frac{d^{p\left(m-k\right)}}{p\left(m-k\right)}\int_{\Omega}\left|\nabla^{m}\mathbf{u}\right|^{p}dx\text{ for all }\mathbf{u}\in W_{0}^{m,p}(\Omega),0\le k\le m-1.\] 
\end{theorem} 
\subsection{Korn's identity}
\begin{lemma}\cite[Lemma A.5]{LR14}\label{lm:Korn} In the context of the introduction it holds that \[\int_{\Omega_{\eta\left(t\right)}}\left(\nabla \mathbf{u}\right)^{T}:\nabla \mathbf{\xi} dx=0\] for all functions $\mathbf{\xi}$ with $\text{tr}_{\eta} \mathbf{\xi}\circ\phi_\eta=b \nu$ for some scalar function $b$; in particular for $\mathbf{\xi}=\mathbf{u}$.

\end{lemma}
\subsection{The Sch\"{a}ffer's/ Leray-Schauder's  fixed point Theorem}
\begin{theorem}\cite[Theorem 4]{Ev10}\label{thm:Schaeffer}
Let $X$ be a Banach space and   $A :  X \mapsto X$ be a continuous and compact mapping such that the set \[\left\{ x\in X:x=\lambda Ax\ \text{for some}\ \lambda\in\left[0,1\right]\right\} \]
is bounded. Then $A$ possesses a fixed point.
\end{theorem}


\subsection{The Piola mapping}
We follow \cite[Remark 2.5]{LR14} and state the following
\begin{lemma}\label{lm:Piola}
 For any $\delta \in C^{2}(\omega)$ with $\left\Vert \delta\right\Vert _{L^{\infty}\left(\omega\right)}<\kappa$ and $\varphi: \omega \mapsto\mathbb{R}^{3}$ we define the    \emph{Piola transform} of $\varphi$ under $\psi_\delta$ by \[\mathcal{J}_{\delta}\varphi:=\left(d\psi_{\delta}\left(\det d\psi_{\delta}\right)^{-1}\varphi\right)\circ\psi_{\delta}^{-1}.\]
Then  $\mathcal{J}_{\delta}$ defines an isomorphism between the Lebesgue and Sobolev spaces on $\Omega$ and the corresponding ones on $\Omega_\delta$ which  moreover preserves the zero boundary values and the divergence-free condition.
\end{lemma}
\begin{proof} It can be shown that 
$\partial_{y}\mathcal{J}_{\delta}\varphi\circ\psi_{\delta}=\left(\det d\psi_{\delta}\right)^{-1}\partial_{i}\varphi\ $ and thus $\text{div}\mathcal{J}_{\delta}\varphi\circ\psi_{\delta}=\left(\det d\psi_{\delta}\right)^{-1}\text{div}\varphi$.
\end{proof}

\subsection{The trace operator}
In order to  rigorously justify the traces  of  functions  defined on $\Omega_{\eta(t)}$ we use 
\begin{lemma}\label{lm:trace} \cite[Corollary 2.9]{LR14} If $1 < p \le \infty$ and $\eta \in H^{2}(\omega)$ with $\left\Vert \eta\right\Vert _{L^{\infty}\left(\omega\right)}<\kappa$, then for any $r \in (1,p)$ the mapping \[\text{tr}_{\eta}:W^{1,p}\left(\Omega_{\eta}\right)\mapsto W^{1-\frac{1}{r},r}\left(\omega\right),\quad\text{tr}_{\eta}\left(v\right):=\left(v\circ\psi_{\eta}\right)_{|\omega}\] is well defined and continuous, with continuity constant depending on $\Omega, r, p$ and a bound for $\left\Vert \eta\right\Vert _{H^{2}\left(\omega\right)}$ and $\tau(\eta)$.
\end{lemma}
  
\subsection{A divergence-free extension operator}
\begin{proposition}\cite[Proposition 3.3]{MS19}\label{prop:solenoidal-extension-operator}
Let  $\eta \in L^{\infty}(I; W^{1,2}(\omega))$ be  with $\left\Vert \eta\right\Vert _{L^{\infty}\left(I\times\omega\right)}<\alpha<\kappa$ and denote $S_{\alpha}:=\left\{ x:\text{dist}\left(x,\partial\Omega\right)<\alpha\right\} $.
 Then  there exists linear and bounded operators 
 \[
\mathcal{M}_{\eta}:L^{1}\left(\omega\right)\mapsto\mathbb{R},\quad\mathcal{F}_{\eta}:\left\{ \xi\in L^{1}\left(I;W^{1,1}\left(\omega\right)\right):\int_{\omega}\xi dx'=0\right\} \mapsto L^{1}\left(I;W_{\text{div}}^{1,1}\left(\Omega\cup S_{\alpha}\right)\right)
 \]
such that for all $\xi \in L^\infty(I; W^{2,2}(\omega))\cap W^{1,\infty}(I;L^2(\omega))$ we have that
\begin{equation}
\begin{aligned}\mathcal{F}_{\eta}\left(\mathcal{M}_{\eta}\left(\xi\right)\right)\in L^{\infty}\left(I;L^{2}\left(\Omega_{\eta}\right)\right)\cap L^{2}\left(I;W_{\text{div}}^{1,2}\left(\Omega_{\eta}\right)\right)\\
\mathcal{M_{\eta}}\left(\xi\right)\in L^{\infty}\left(I;W^{2,2}\left(\omega\right)\right)\cap W^{1,\infty}\left(I;L^{2}\left(\omega\right)\right)\\
\text{tr}_{\eta}\mathcal{F}_{\eta}\left(\mathcal{M_{\eta}}\left(\xi\right)\right)=\mathcal{M_{\eta}}\left(\xi\right)\mathbf{\nu}\\
\mathcal{F}_{\eta}\left(\mathcal{M}\left(\xi\right)\right)=0\quad\text{on}\ I\times\Omega\setminus S_{\alpha}
\end{aligned}
\end{equation}
  and furthermore for all $p \in (1,\infty)$ and $q\in (1, \infty]$ we have
  
   \[
\left\Vert \mathcal{F}_{\eta}\left(\mathcal{M}\left(\xi\right)\right)\right\Vert _{L^{q}\left(I;W^{1,p}\left(\Omega\cup S_{\alpha}\right)\right)}\lesssim\left\Vert \xi\right\Vert _{L^{q}\left(I;W^{1,p}\left(\omega\right)\right)}+\left\Vert \xi\nabla\eta\right\Vert _{L^{q}\left(I;L^{p}\left(\omega\right)\right)}
   \] and
    \[
 \left\Vert \partial_{t}\mathcal{F}_{\eta}\left(\mathcal{M}_{\eta}\left(\xi\right)\right)\right\Vert _{L^{q}\left(I;L^{p}\left(\Omega\cup S_{\alpha}\right)\right)}\lesssim\left\Vert \partial_{t}\xi\right\Vert _{L^{q}\left(I;L^{p}\left(\omega\right)\right)}+\left\Vert \xi\partial_{t}\eta\right\Vert _{L^{q}\left(I;L^{p}\left(\omega\right)\right)}.
    \]
 The continuity constant depends only on $\Omega$, $p$, and $\kappa$.
 \end{proposition}
  
\subsection{ Sobolev embedding on moving domains}
\begin{lemma}\cite[Corollary 2.10]{LR14}
\label{lm:Sobolev-embedding}
If $1 < p<3$ and $\eta \in H^{2}(\omega)$ with $\left\Vert \eta\right\Vert _{L^{\infty}\left(\omega\right)}<\kappa$ then \[W^{1,p}\left(\Omega_{\eta}\right)\hookrightarrow L^{p^{*}}\left(\Omega_{\eta}\right)\ \text{with} \   p^{*}:=\frac{3p}{3-p}\] and the embedding constant depends on $\Omega$, $p$ and a bound for $\left\Vert \eta\right\Vert _{H^{2}\left(\omega\right)}$ and $\tau (\eta)$.
\end{lemma}
\subsection{The Kakutani-Glicksberg-Fan fixed point theorem}
We will need the following set-valued fixed-point result which can be found in \cite{G52} or in  \cite[Chapter 2, Section 5.8]{GD03}
\begin{theorem}[Kakutani-Glicksberg-Fan]\label{thm: Kakutain} 
Let $C$ be a convex subset of a normed vector space $Z$ and let $F:C \to \mathcal{P}(C)$ be a  set-valued mapping which has closed graph. Moreover, let $F(C)$ be contained in a compact subset of $C$, and let $F(z)$ be non-empty, convex, and compact for all $z \in C$. Then $F$ possesses a fixed point, that is there is $c_0 \in C$ with $c_0 \in F(c_0)$. 
\end{theorem}
We say a set-valued mapping $F : C \mapsto \mathcal{P}(C)$ has \emph{closed graph} provided that the set $\left\{ \left(x,y\right):y\in F\left(x\right)\right\} $ is closed in $X \times Y$ with the product topology or that, equivalently, for any sequences $x_n \to x$ and $y_n \to y$ with $y_n \in F(x_n)$ for any $ \ge 1$ it follows that $y \in F(x)$.


\end{document}